\documentclass[11pt]{article}

\usepackage{amssymb,latexsym}
\usepackage{amsmath,amscd}
\usepackage{theorem}
\usepackage[all]{xy}
\usepackage{url}
\usepackage{stmaryrd}
\usepackage{bm} 
\usepackage{graphics}
\usepackage{adjustbox}

\usepackage{enumitem} 
\setlist[enumerate]{itemsep=0em}

\title{\bf On hereditarily self-similar \\$p$-adic analytic pro-$p$ groups}

\author{
        Francesco Noseda  \thanks{{\tt noseda@im.ufrj.br}.}
        \\[0.1cm]
        Ilir Snopce \thanks{{\tt ilir@im.ufrj.br}. Supported by the Alexander von Humboldt Foundation, 
        CAPES (grant 88881.145624/2017-01), FAPERJ and CNPq.
        }
        \\[0.2cm]
        \footnotesize{Mathematics Institute - 
        Federal University of Rio de Janeiro,}\\
        \footnotesize{Avenida Athos da Silveira Ramos 149, 
        }\\
        \footnotesize{21941-909, Rio de Janeiro,  Brazil.}
}

\date{}

\newcommand{\bb}[1]{\mathbb{#1}}
\newcommand{\cl}[1]{\mathcal{#1}}

\newcommand{\mr}[1]{\mathrm{#1}}

\newcommand{\rar}{\rightarrow}

\newcommand{\ol}{\overline}

\newcommand{\vep}{\varepsilon}

\newcommand{\les}{\leqslant}
\newcommand{\ges}{\geqslant}

\newcommand{\ep}{\hfill $\square$} 

\newcommand{\gen}[1]{\langle #1 \rangle}

\newtheorem{lemma} {Lemma} [section]
\newtheorem{proposition} [lemma] {Proposition}

\newtheorem{theorem} [lemma] {Theorem}
\newtheorem{corollary} [lemma] {Corollary}
\newtheorem{definition}[lemma] {Definition}

\theorembodyfont{\rm} 

\newtheorem{remark}[lemma]{Remark}

\newenvironment{proof}{{\sc Proof:}}{
\hfill $\square$}

\numberwithin{equation}{section}

\theorembodyfont{\it}

\newtheorem{theoremx}{Theorem}

\newtheorem{conjx}[theoremx]{Conjecture}

\usepackage{array}
\usepackage{appendix}

\usepackage[colorlinks]{hyperref}

\begin{document}

\maketitle


\begin{abstract}
A non-trivial finitely generated pro-$p$ group $G$ is said to be 
strongly hereditarily self-similar of index $p$
if every non-trivial finitely generated closed subgroup of $G$ admits a 
faithful self-similar action on a $p$-ary tree. 
We classify the solvable torsion-free 
$p$-adic analytic pro-$p$ groups of dimension less than $p$  
that are strongly hereditarily self-similar of index $p$.
Moreover, we show that a solvable torsion-free $p$-adic
analytic pro-$p$ group of dimension less than $p$ is strongly 
hereditarily self-similar of index $p$ if and only if 
it is isomorphic to the maximal pro-$p$ Galois group of some 
field that contains a primitive 
$p$-th root of unity. 
As a key step for the proof of the above results, 
we classify the 3-dimensional solvable 
torsion-free 
$p$-adic analytic pro-$p$ groups that admit a faithful self-similar  
action on a $p$-ary tree, completing the classification
of the 3-dimensional torsion-free 
$p$-adic analytic pro-$p$ groups that admit such actions.
\end{abstract}

\let\thefootnote\relax\footnotetext{\textit{Mathematics Subject Classification (2020): }
Primary 20E18, 22E20, 20E08; Secondary 22E60.}
\let\thefootnote\relax\footnotetext{\textit{Key words:} self-similar group, 
pro-$p$ group, $p$-adic analytic group, $p$-adic Lie lattice,
maximal Galois pro-$p$ group.}

\tableofcontents

\section*{Introduction}

Groups that admit a faithful self-similar action on some regular 
rooted $d$-ary 
tree $T_d$ 
form an interesting class that contains many important examples 
such as the Grigorchuk 
2-group \cite{Gri80}, the Gupta-Sidki $p$-groups \cite{GuSi83}, the affine groups 
$\mathbb{Z}^n \rtimes GL_n(\mathbb{Z})$ 
\cite{BrSi98}, and
groups obtained as iterated monodromy groups of self-coverings 
of the Riemann sphere by post-critically finite rational maps \cite{NekSSgrp}.
Recently there has been an intensive study on the self-similar actions of other 
important families of groups including abelian groups \cite{BrSi10}, 
wreath products of abelian groups \cite{DSwreath},
finitely 
generated nilpotent groups \cite{BeSi07},  arithmetic groups \cite{Ka12}, and 
groups of type $\mr{FP}_n$ \cite{KoSi19}. Self-similar actions of some classes of finite $p$-groups were studied 
in \cite{Su11} and \cite{BaFaFeVa}.  

We say that a group $G$ is \textbf{self-similar of index} $\bm{d}$ 
if $G$ admits a faithful self-similar action on
$T_d$ that is transitive on the first level; moreover,
we say that $G$ is \textbf{self-similar} if it is self-similar of 
some index $d$. In \cite{NS2019} 
we initiated the study of self-similar actions of $p$-adic analytic pro-$p$ groups. 
In particular,  we classified the 
3-dimensional \textit{unsolvable} torsion-free $p$-adic analytic pro-$p$ groups for $p\ges 5$, 
and determined which of them admit a faithful self-similar action on a $p$-ary tree. 
In the present paper, instead, we focus on the study of self-similar actions of torsion-free \textbf{solvable} $p$-adic analytic pro-$p$ groups.

It is fairly easy to show that every free abelian group ${\mathbb{Z}}^r$ of finite rank $r \ges 1$ is self-similar of any index $d\ges 2$
(cf. \cite[Section 2.9.2]{NekSSgrp}; see also \cite{NSAutBinTree}) 
Hence, every non-trivial  subgroup of ${\mathbb{Z}}^r$  is self-similar of any index $d \ges 2$.  Similarly, every non-trivial closed subgroup of a free abelian pro-$p$ group ${\mathbb{Z}}_p^r$ is self-similar of index $p^k$, for $k \ges 1$. 
Motivated by this phenomenon we make the following  definitions.  
A finitely generated pro-$p$ group $G$ is said to be 
\textbf{hereditarily self-similar of index $\bm{p}^k$}
if any open subgroup of $G$
is self-similar of index $p^k$. If $G$ and all of its non-trivial 
finitely generated closed subgroups are self-similar of index 
$p^k$ then $G$ is said to be 
\textbf{strongly hereditarily self-similar of index $\bm{p}^k$}. 

From \cite[Proposition 1.5]{NS2019}, it follows that any torsion-free
$p$-adic analytic pro-$p$ group of dimension 1 or 2 
is strongly hereditarily self-similar of index $p^k$
for all $k\ges 1$.
Moreover, it is not difficult to see that if $p\ges 5$ then 
any 3-dimensional \textit{solvable} torsion-free $p$-adic analytic pro-$p$ group
is strongly hereditarily self-similar of index $p^{2m}$ for all
$m\ges 1$ (see Proposition \ref{pgstrong2m}).
Observe that the latter class  contains
a continuum of groups that are pairwise incommensurable
(see \cite{SnopceJGT2016}),
in contrast to the discrete case, where there are only countably many pairwise
non-isomorphic finitely generated self-similar groups (cf. \cite[Section 1.5.3]{NekSSgrp}).
On the other hand, 
it is an interesting problem to understand which pro-$p$ groups have
the property of being strongly
hereditarily self-similar of  \textbf{index $\bm{p}$},
and the main result of this paper is the classification of 
the  solvable torsion-free $p$-adic analytic pro-$p$ groups with this property.

\begin{theoremx}\label{tmainresult}
Let $p$ be a prime, and 
let $G$ be a solvable torsion-free $p$-adic analytic pro-$p$ group.
Suppose that $p> d:= \textrm{dim}(G)$. 
Then $G$ is strongly hereditarily self-similar of index $p$ if and only if  
$G$ is isomorphic to one of the following groups:
  \begin{enumerate}
  \item For $d\ges 1$, the abelian pro-$p$ group $\mathbb{Z}_p^d$;
  \item  
  For $d\ges 2$, the metabelian pro-$p$ group
  $G^d(s):=\bb{Z}_p\ltimes \bb{Z}_p^{d-1}$, 
  where
  the canonical generator of 
  $\bb{Z}_p$ acts on $\bb{Z}_p^{d-1}$
  by  multiplication by the scalar $1+p^s$, for some integer $s\ges 1$.
  \end{enumerate} 
\end{theoremx}

\noindent
Observe that the `if' part of the theorem holds in greater generality
(Proposition \ref{here2}).
It is worth noting that during the last decade the groups listed in 
Theorem \ref{tmainresult}
have appeared in the literature  in different contexts (see, for example, \cite{KSjalg11}, \cite{Quad14}, \cite{KSquart14}, \cite{Sn15} and \cite{Sn09}).
The reader will find a more detailed account
of the related results at the end of Section \ref{sresgp}.

Let $K$ be a field. The absolute Galois group of $K$ is the profinite group $G_K=\textrm{Gal}(K_s/K)$, where $ K_s$ is a separable closure of $K$. The maximal pro-$p$ Galois group of $K$, denoted by $G_K(p)$, is the maximal pro-$p$ quotient of $G_K$.  More precisely, $G_K(p) =  \textrm{Gal}(K(p)/K)$, where  $K(p)$ is 
the composite of all finite Galois $p$-extensions of $K$ (inside $K_s$). Describing absolute Galois groups of fields among profinite groups is one of the most important problems in Galois theory. Already describing  $G_K(p)$  among pro-$p$ groups is a remarkable challenge. Theorem \ref{tmainresult} and a result of Ware \cite{Ware92} yield the following.

\begin{theoremx}\label{tmainB}
Let $p$ be a prime, and let $G$ be a non-trivial 
solvable torsion-free $p$-adic analytic pro-$p$ group.
Suppose that $p >  \textrm{dim}(G)$. 
Then $G$ is strongly hereditarily self-similar of index $p$ 
if and only if $G$ is isomorphic to the maximal pro-$p$ Galois 
group of some field that contains a primitive $p$-th root of unity. 
\end{theoremx}

\noindent
Similarly to Theorem \ref{tmainresult}, the `if' part holds
in greater generality (Proposition \ref{maximal}).

The proof of Theorem \ref{tmainresult} is by induction on $d=\mr{dim}(G)$.
As mentioned above, for $d=1,2$ matters are trivial, while for $d=3$
interesting phenomena start to occur.
Indeed, as basis for the induction, one has to consider the case $d=3$,
and this leaded us to the classification result below. 
This result is interesting on its own right since it completes the classification
started by \cite[Theorem B]{NS2019}
of the 3-dimensional torsion-free $p$-adic analytic pro-$p$ groups
that are self-similar of index $p$.

\begin{theoremx}\label{tmaingroups}
Let $p\ges 5$ be a prime, and fix $\rho\in\bb{Z}_p^*$ not a square modulo $p$.
Let $G$ be a 3-dimensional solvable torsion-free $p$-adic analytic pro-$p$ group.
Then the following holds.
\begin{enumerate}
\item $G$ is self-similar of index $p^2$. 
\item Let $L$ be the $\bb{Z}_p$-Lie lattice associated with $G$. 
Then $G$ is self-similar of index $p$ if and only if $L$ is isomorphic
to a Lie lattice presented in the following list (cf. Remark \ref{rclassif};
the parameters below take values $s,r,t\in\bb{N}$, $c\in\bb{Z}_p$ and
$\vep\in\{0,1\}$).
\begin{enumerate}
\item 
$
\langle\, x_0, x_1, x_2   \mid 
[x_1, x_2] =  0,\,  
[x_0, x_1] = 0,\, 
[x_0, x_2] =  0 \,
\rangle $.
\item 
For $s\ges 1$, 
$
\langle\, x_0, x_1, x_2   \mid 
[x_1, x_2] =  0,\,  
[x_0, x_1] = p^sx_1,\, 
[x_0, x_2] =  p^sx_2 \,
\rangle$.
\item 
For $s,r\ges 1$ and $v_p(c)=1$,
$$
\langle\, x_0, x_1, x_2   \mid 
[x_1, x_2] =  0,\,  
[x_0, x_1] = p^sx_1 + p^{s+r}cx_2,\, 
[x_0, x_2] =  p^{s+r}x_1 + p^sx_2 \,
\rangle .
$$
\item 
$
\langle\, x_0, x_1, x_2   \mid 
[x_1, x_2] =  0,\,  
[x_0, x_1] = p^{s+1}\rho^\vep x_2,\, 
[x_0, x_2] =  p^{s}x_1  \,
\rangle
$.
\item For $s\ges 1$,
$
\langle\, x_0, x_1, x_2   \mid 
[x_1, x_2] =  0,\,  
[x_0, x_1] = p^{s} x_2,\, 
[x_0, x_2] =  p^{s}x_1  \,
\rangle$.
\item For $r\ges 1$ and $v_p(c)= 1$,
$$
\langle\, x_0, x_1, x_2   \mid 
[x_1, x_2] =  0,\,  
[x_0, x_1] = p^{s+r}x_1+p^scx_2,\, 
[x_0, x_2] =  p^{s}x_1  \,
\rangle.$$
\item For $s\ges 1$ and $v_p(1+4c)=1$,
$$
\langle\, x_0, x_1, x_2   \mid 
[x_1, x_2] =  0,\,  
[x_0, x_1] = p^{s}x_1+p^scx_2,\, 
[x_0, x_2] =  p^{s}x_1  \,
\rangle.$$
\end{enumerate}
\end{enumerate}
\end{theoremx}

In dimension 3, Theorem \ref{tmaingroups} and  \cite[Theorem B]{NS2019} yield the following stronger version of Theorem  \ref{tmainresult}.

\begin{theoremx}\label{tmainD}
Let $p\ges 5$ be a prime, and let $G$ be a 
3-dimensional torsion-free $p$-adic analytic pro-$p$ group.
Then the following are equivalent.
\begin{enumerate}
\item $G$ is hereditarily self-similar of index $p$.
\item $G$ is strongly hereditarily self-similar of index $p$.
\item $G$ is isomorphic  to $\bb{Z}_p^3$ or to $G^3(s)$ for some integer $s\ges 1$.
\end{enumerate}
\end{theoremx}

\vspace{3mm}

We believe that one can drop the assumption of solvability in Theorem \ref{tmainresult}  even in higher dimension.

\begin{conjx}\label{conjE}
Let $p$ be a prime, and 
let $G$ be a torsion-free $p$-adic analytic pro-$p$ group of dimension $d$. 
Suppose that $p> d$. 
Then $G$ is strongly hereditarily self-similar of index $p$ if and only if $G$ is isomorphic  to $\bb{Z}_p^d$ for $d \ges 1$ or to 
$G^d(s)$ for $d\ges 2$ and some integer $s\ges 1$.
\end{conjx}

\vspace{3mm}

\noindent 
\textbf{Main strategy and outline of the paper.} 
For the proof of the main results we use Lie methods.
More precisely, we use the language of 
virtual endomorphisms (see, for instance, \cite[Proposition 1.3]{NS2019}) 
to translate self-similarity problems on 
$p$-adic analytic groups to problems on $\bb{Z}_p$-Lie lattices 
(Proposition \ref{Prop-NS2}). Recall from \cite{NS2019} that
a $\bb{Z}_p$-Lie lattice $L$
is said to be self-similar of index $p^k$ if there exists
a homomorphism of algebras $\varphi:M\rar L$ where $M\subseteq L$ 
is a subalgebra of index $p^k$ and $\varphi$ is simple, which means
that there are no non-zero ideals of $L$ that are $\varphi$-invariant.

In \textbf{Section \ref{seclie}} we prove results on Lie lattices,
and for the main ones mentioned here we assume $p\ges 3$. 
The first main result of that section is \textbf{Theorem \ref{smain3}}, where
 we classify the 3-dimensional solvable $\bb{Z}_p$-Lie lattices
that are self-similar of index $p$, complementing the analogue
result for unsolvable lattices proven in \cite[Theorem 2.32]{NS2019}.
In Definition \ref{dhss} we introduce the notion of (strongly)
hereditarily self-similar Lie lattice. 
Thanks to the classification result, we are able to prove 
\textbf{Proposition \ref{hereddim3v2}}, which is a classification
of the 3-dimensional $\bb{Z}_p$-Lie lattices that are (strongly)
hereditarily self-similar of index $p$.
This result is particularly relevant since it is used as the basis
of the induction (which is on dimension) for the proof of the second main result
on Lie lattices, \textbf{Theorem \ref{tshssp}}, which
provides a classification of the solvable $\bb{Z}_p$-Lie lattices 
that are strongly hereditarily self-similar of index $p$.
At the beginning of Section \ref{seclie} the reader will find 
a more detailed account of its structure.

In \textbf{Section \ref{sresgp}} we prove the main theorems of the paper, 
and provide additional results on hereditarily self-similar groups.
We observe that 
Theorem \ref{tmaingroups} follows from Theorem \ref{smain3},
Theorem \ref{tmainD} follows from Proposition \ref{hereddim3v2}, and  
Theorem \ref{tmainresult} follows from Theorem \ref{tshssp}.
In \textbf{Section \ref{sopen}} we state some
open problems that we consider challenging
and that we believe will stimulate future research on the subject.

\vspace{5mm}   
   
\noindent 
\textbf{Notation.} Throughout the paper, $p$ denotes a prime number,
and $\equiv_p$ denotes equivalence modulo $p$. 
For $p\ges 3$ we fix $\rho\in\bb{Z}_p^*$ not a square modulo $p$.
We denote  the $p$-adic valuation
by  $v_p:\bb{Q}_p\rar \bb{Z}\cup\{\infty\}$. The set $\bb{N}$ 
of natural numbers is assumed to contain $0$.
For the lower central series $\gamma_n(G)$ and the derived series 
$\delta_n(G)$ of a group 
(or Lie algebra) $G$ we use
the convention $\gamma_0(G)=G$ and $\delta_0(G)=G$.
By a $\bb{Z}_p$-lattice we mean a finitely generated free $\bb{Z}_p$-module.
Let $L$ be a $\bb{Z}_p$-lattice. 
When $M\subseteq L$ is a submodule, we denote
the isolator of $M$ in $L$ by 
$\mr{iso}_L(M):=\{x\in L:\, \exists\,k\in\bb{N}\quad p^kx\in M\}$.
We denote by $\gen{x_1,...,x_n}$ the \textit{submodule} of $L$ generated
by $x_1,...,x_n\in L$.
When $L$ has the structure of a Lie algebra, we denote its center
by $Z(L)$.

\vspace{5mm}

\noindent
\textbf{Aknowledgements.}
The second author
thanks the Heinrich Heine University in D\"usseldorf for its hospitality.


\section{Results on Lie lattices}\label{seclie}

In this section, which is self-contained, we prove 
results about self-similarity of $\bb{Z}_p$-Lie lattices. 
The main results, mentioned in the Introduction, are proved
under the assumption that $p\ges 3$. On the other hand,
most of the auxiliary results are valid and proved
for any $p$, and we believe that they constitute a large part of the
work needed to generalize the main results to $p=2$.
The structure of the section is as follows.
In Section \ref{smet} we prove some basic results on
$\bb{Z}_p$-Lie lattices that admit an abelian
ideal of codimension $1$;
these results are used both for the study of 3-dimensional lattices
and of lattices in higher dimension.
After two preparatory technical sections (Sections \ref{sssres} and \ref{snssres3}), 
in Section \ref{smain3} we prove the main
Theorem \ref{talgpodd}.
After another preparatory section (Section \ref{snssresh}),
in Section \ref{sechss} we prove the main 
Proposition \ref{hereddim3v2} and Theorem \ref{tshssp}.
Apart from the main results, a few statements are worth to be
mentioned here, for instance, Propositions \ref{pabelianid},
\ref{pll1nssp}, and \ref{heranyindex}.
The most difficult technical results are the proofs of non-self-similarity 
of Propositions \ref{pl7nssp} and \ref{phdn}.
\\

We will be dealing with several families of $\bb{Z}_p$-Lie lattices, which
we list in the definition below. For $p\ges 3$, families from (\ref{C0})
to (\ref{C5}) are needed for the classification of 3-dimensional solvable
$\bb{Z}_p$-Lie lattices (see Remark \ref{rclassif}).
Family (\ref{C6}) generalizes families (\ref{C0}) and (\ref{C1}), while family
(\ref{C7}) generalizes families (\ref{C4}) and (\ref{C5}).

\begin{definition}\label{dlielat}
We define eight families of 3-dimensional solvable $\bb{Z}_p$-Lie lattices through 
presentations.
\begin{enumerate}
\setcounter{enumi}{-1}
\item \label{C0}
$L_{\ref*{C0}}= 
\langle\, x_0, x_1, x_2   \mid 
[x_1, x_2] =  0,\,  
[x_0, x_1] = 0,\, 
[x_0, x_2] =  0 \,
\rangle $.
\item \label{C1}
For $s\in\bb{N}$, 
$L_{\ref*{C1}}(s)= 
\langle\, x_0, x_1, x_2   \mid 
[x_1, x_2] =  0,\,  
[x_0, x_1] = p^sx_1,\, 
[x_0, x_2] =  p^sx_2 \,
\rangle $.
\item \label{C2}
For $s,r\in\bb{N}$ and 
$c\in\bb{Z}_p$,
$$L_{\ref*{C2}}(s,r,c)=
\langle\, x_0, x_1, x_2   \mid 
[x_1, x_2] =  0,\,  
[x_0, x_1] = p^sx_1 + p^{s+r}cx_2,\, 
[x_0, x_2] =  p^{s+r}x_1 + p^sx_2 \,
\rangle .
$$
\item \label{C3}
For $s\in\bb{N}$, 
$L_{\ref*{C3}}(s)=
\langle\, x_0, x_1, x_2   \mid 
[x_1, x_2] =  0,\,  
[x_0, x_1] = 0,\, 
[x_0, x_2] =  p^sx_1 \,
\rangle $.
\item \label{C4}
For $p\ges 3$, $s,t\in\bb{N}$ and $\vep\in\{0,1\}$,
$$L_{\ref*{C4}}(s,t,\vep)= 
\langle\, x_0, x_1, x_2   \mid 
[x_1, x_2] =  0,\,  
[x_0, x_1] = p^{s+t}\rho^\vep x_2,\, 
[x_0, x_2] =  p^{s}x_1  \,
\rangle .
$$
\item \label{C5}
For $s,r\in\bb{N}$ and $c\in \bb{Z}_p$,
$$L_{\ref*{C5}}(s,r,c)= 
\langle\, x_0, x_1, x_2   \mid 
[x_1, x_2] =  0,\,  
[x_0, x_1] = p^{s+r}x_1+p^scx_2,\, 
[x_0, x_2] =  p^{s}x_1  \,
\rangle .$$
\item \label{C6}
For $a\in\bb{Z}_p$, 
$L_{\ref*{C6}}(a)= 
\langle\, x_0, x_1, x_2   \mid 
[x_1, x_2] =  0,\,  
[x_0, x_1] = ax_1,\, 
[x_0, x_2] =  ax_2 \,\rangle $.
\item \label{C7}
For $s\in\bb{N}$ and $a,c\in\bb{Z}_p$,
$$L_{\ref*{C7}}(s,a,c)= 
\langle\, x_0, x_1, x_2   \mid 
[x_1, x_2] =  0,\,  
[x_0, x_1] = p^sax_1 + p^{s}c x_2, \, 
[x_0, x_2] =  p^{s}x_1 \,
\rangle .
$$
\end{enumerate}
\end{definition}

\subsection{On a class of metabelian Lie lattices}\label{smet}
Given an integer $d\ges 1$, 
we are going to consider $(d+1)$-dimensional
$\bb{Z}_p$-Lie lattices that admit a $d$-dimensional abelian ideal.
Greek indices will take values in $\{0,...,d\}$, 
while Latin indices will take values in $\{1,..,d\}$. 
For matrices in $gl_{d+1}(\bb{Q}_p)$ we use a notation like 
$\ol{U}=(U_{\alpha\beta})$; moreover, 
for such a matrix, we denote $U=(U_{ij})\in gl_d(\bb{Q}_p)$.

Let $L$ be a $(d+1)$-dimensional antisymmetric $\bb{Z}_p$-algebra.
Observe that $L$ admits a $d$-dimensional abelian ideal
if and only if there exists a basis $\bm{x}=(x_0,..,x_d)$ of $L$
and a matrix $A\in gl_d(\bb{Z}_p)$, $A=(A_{ij})$, such that 
for all $i,j$ we have 
$$
\left\{
\begin{array}{l}
[x_i,x_j] = 0\\[3pt]
[x_0,x_i] = \sum_{l} A_{li}x_l.
\end{array}
\right.
$$
In this case, $\gen{x_1,...,x_d}$ is a $d$-dimensional abelian ideal.
It is immediate to see that, for such a $L$, the Jacobi identity holds, and 
that $\delta_2(L)=\{0\}$; in other words, $L$ is a metabelian Lie lattice. 
When it exists, a basis as above is called a \textbf{good basis}
of $L$, and $A$ is called the \textbf{matrix of} $L$ with respect to the (good) 
basis $\bm{x}$. Observe that $A$ is the matrix of the homomorphism of 
lattices $[x_0, \cdot\,]: \gen{x_1,...,x_d}\rar \gen{x_1,...,x_d}$
with respect to the displayed bases.

Let $L$ be a $(d+1)$-dimensional $\bb{Z}_p$-Lie lattice that admits a 
$d$-dimensional abelian ideal, let $\bm{x}$ be a good basis of $L$, and let $A$ be
the corresponding matrix. Observe that $\mr{rk}(A)=\mr{dim}[L,L]$, so that
$\mr{rk}(A)$ is an isomorphism invariant of $L$. In particular, $A$ is invertible
over $\bb{Q}_p$ if and only if $\mr{dim}[L,L]=d$, a relevant special case.
Let $M\subseteq L$ be a finite-index submodule,
let $\bm{y}=(y_0,...,y_d)$ be a basis of $M$, and let 
$\ol{U}=(U_{\alpha\beta})\in gl_{d+1}(\bb{Z}_p)$ be the 
matrix of $\bm{y}$ with respect to $\bm{x}$, namely, 
$y_\beta = \sum_{\alpha}U_{\alpha\beta}x_\alpha$.
Observe that $M\cap \gen{x_1,...,x_d}=\gen{y_1,...,y_d}$
if and only if $U_{0i}=0$ for all $i$; moreover, there exists a basis of $M$
such that $U_{\alpha\beta}=0$
for all $\alpha<\beta$. We also observe that $\mr{dim}[M,M]=\mr{dim}[L,L]$.

\begin{lemma}\label{lsmetB}
Let $d, L, \bm{x}, A, M, \bm{y}, \ol{U}=(U_{\alpha\beta})$ be as above. Assume
$M\cap \gen{x_1,...,x_d}=\gen{y_1,...,y_d}$.
Then $U=(U_{ij})$ is invertible over $\bb{Q}_p$ (it is a $d\times d$ matrix),
and we may define $B\in gl_d(\bb{Q}_p)$ by $B= U_{00}U^{-1}AU$.
Then the following holds.
\begin{enumerate}
\item $M$ is a subalgebra of $L$ if and only if
$B$ has entries in $\bb{Z}_p$.
\item Assume that $M$ is a subalgebra of $L$.
Then $\bm{y}$ is a good basis of $M$ and $B$ is the matrix of 
$M$ with respect to $\bm{y}$.
\end{enumerate}
\end{lemma}

\begin{proof}
Since $y_j= \sum_i U_{ij} x_i$, it follows that $[y_i,y_j]=0$. Over $\bb{Q}_p$, we have 
$$[y_0,y_j]= 
U_{00}\sum_i U_{ij}[x_0,x_i] = 
U_{00}\sum_{i,l} U_{ij}A_{li}x_l =
U_{00}\sum_{i,l,k} U_{ij}A_{li}U^{-1}_{kl}y_k,
$$
so that $[y_0,y_j]=\sum_k B_{kj}y_k$. The lemma follows.
\end{proof}

\vspace{5mm}

\noindent
Observe that the case $M=L$ is included in the above discussion. In this
case, $U$ is invertible over $\bb{Z}_p$, and the defining formula 
of $B$ is the change-of-basis formula for the matrix of $L$
(under lower block-triangular changes of basis).

We now study homomorphisms of algebras. 

\begin{lemma} \label{lmetmor}
Let $L,M$ be $(d+1)$-dimensional 
$\bb{Z}_p$-Lie lattices endowed with good bases $\bm{x},\bm{y}$,
and let $A,B$ be the respective matrices.
Let $\varphi: M\rar L$ be a homomorphism of modules, and let 
$\ol{F}\in gl_{d+1}(\bb{Z}_p)$ be the matrix of $\varphi$ with respect
to the given bases, namely, 
$\varphi(y_\beta)= \sum_\alpha F_{\alpha\beta} x_\alpha$.
Then the following holds.
\begin{enumerate}
\item \label{lmetmor1}
The homomorphism $\varphi$ 
is a homomorphism of algebras if and only if, 
for all $i,j$:
 \begin{enumerate}
 \item \label{lmetmor1a} 
 $\sum_{l} F_{0l}B_{lj}=0$.
 \item \label{lmetmor1b}
 $F_{0i} (AF)_{kj}-F_{0j}(AF)_{ki} = 0$, for all $k$.
 \item \label{lmetmor1c}
 $(FB)_{ij} = F_{00}(AF)_{ij} - F_{0j}\sum_{l} A_{il}F_{l0}$.
 \end{enumerate}
\item \label{lmetmor2}
Assume $\varphi$ is a homomorphism of algebras and 
$\mr{dim}\,[M,M]=d$. Then $F_{0i}=0$ for all $i$.
\item \label{lmetmor3} 
Assume $F_{0i}=0$ for all $i$. Then $\varphi$ is a homomorphism of algebras
if and only if $FB=F_{00}AF$. 
\end{enumerate}
\end{lemma}

\begin{proof}
The homomorphism $\varphi$ 
is a homomorphism of algebras if and only if, for all $i,j$, 
$[\varphi(y_i),\varphi(y_j)]=0$ and 
$\varphi([y_0,y_j]) = [\varphi(y_0),\varphi(y_j)]$. 
One computes, 
\begin{eqnarray*}
[\varphi(y_i),\varphi(y_j)] &=&
\sum_{l} (F_{0i}F_{lj}-F_{0j}F_{li})[x_0,x_l]=
\sum_{k} (F_{0i}(AF)_{kj}-F_{0j}(AF)_{ki})x_k,\\
\varphi([y_0,y_j]) & = &  \sum_{l,\alpha} B_{lj}F_{\alpha l}x_\alpha =
\left(\sum_{l}F_{0l}B_{lj}\right)x_0 + \sum_i (FB)_{ij}x_i,\\[5pt]
[\varphi(y_0),\varphi(y_j)]&=& 
\sum_l (F_{00}F_{lj}-F_{l0}F_{0j}) [x_0,x_l] =
\sum_i \left(F_{00}(AF)_{ij}+ F_{0j}\sum_l A_{il}F_{l0}\right)x_i,
\end{eqnarray*}
from which item (\ref{lmetmor1}) follows.
For item (\ref{lmetmor2}), one observes
that $B$ is invertible over $\bb{Q}_p$ and applies
item (\ref{lmetmor1a}). 
Item (\ref{lmetmor3}) follows directly from
item (\ref{lmetmor1}).
\end{proof}

\begin{corollary}\label{cFeqs}
Let $L$ be a $(d+1)$-dimensional $\bb{Z}_p$-Lie lattice with $\mr{dim}[L,L]=d$,
and let $\varphi:M\rar L$ be a virtual endomorphism of $L$. 
Let $\bm{x}$ be a good basis of $L$. Then the following holds.
\begin{enumerate}
\item \label{cFeqs1}
Let $\bm{y}$ be a basis 
of $M$ with the property $M\cap \gen{x_1,...,x_d}=\gen{y_1,...,y_d}$.
Then $FB=F_{00}AF$, where $A$, $B$ and $\ol{F}$ are as in Lemma \ref{lmetmor}.
\item \label{cFeqs2} 
Assume $\gen{x_1,..,x_d}\subseteq M$. Then $\gen{x_1,...,x_d}$ is
a $\varphi$-invariant ideal of $L$.
\end{enumerate}
\end{corollary}

\begin{remark}\label{rsol2di}
Any 3-dimensional solvable $\bb{Z}_p$-Lie lattice admits a 
2-dimensional abelian ideal.
\end{remark}


\subsection{Self-similarity results}\label{sssres}

When $\varphi:M\rar L$ is a virtual endomorphism of a Lie lattice $L$,
we denote by $D_n$, $n\in \bb{N}$, the domain of the power $\varphi^n$, and we define 
$D_\infty = \bigcap_{n\in\bb{N}} D_n$. We recall that, by definition,
$D_0 = M$ and $D_{n+1}=\{x\in M:\, \varphi(x)\in D_n\}$
(see, for instance, \cite[Definition 1.1]{NS2019}).

\begin{proposition}\label{pabelianid}
Let $k,d\ges 1$ be integers, and let $L$ be a
$\bb{Z}_p$-Lie lattice
of dimension $d+1$. Assume that $L$ admits a $d$-dimensional abelian ideal.
Then $L$ is self-similar of index $p^{dk}$.
\end{proposition}

\begin{proof}
If $L$ is abelian, it is easy to see that $L$ is self similar of index $p^m$
for all $m\ges 1$. Assume that $L$ is not abelian.
There exists a basis $(x_0,x_1,...,x_d)$ of $L$ such that 
$[x_i, x_j]=0$ and $[x_0,x_i]=\sum_{l=1}^d A_{li}x_l$ for all $1\les i,j\les d$,
and some $A_{li}\in\bb{Z}_p$. We define 
$M = \langle x_0, p^{k}x_1,...,p^{k}x_d\rangle$,
and observe that $M$ is a subalgebra of $L$ of index $p^{dk}$.
We define a homomorphism of algebras $\varphi: M \rar L$ by
$\varphi(x_0) = x_0$ and $\varphi(p^{k}x_i)= x_i$ for $1\les i\les d$.
We are going to show that $\varphi$ is simple.
One sees that $D_\infty =\langle x_0 \rangle$. Let $I$ be a non-trivial ideal of $L$.
We show that $L$ is not $\varphi$-invariant by proving the existence of $w\in I$
such that $w\not\in D_\infty$. Indeed, there exists $0\neq z =a_0x_0+...+a_dx_d\in I$.
If $a_i\neq 0$ for some $i>0$ then one may take $w=z$. Otherwise $z = a_0x_0$
with $a_0\neq 0$. Since $L$ is not abelian, there exists $i>0$
such that $[x_0,x_i]\neq 0$. In this case one may take $w = [z,x_i]$.
\end{proof}

\begin{corollary}\label{csigmap2}
Let $k\ges 1$ be an integer, and let $L$ be a $3$-dimensional
solvable $\bb{Z}_p$-Lie lattice. Then $L$ is self-similar of index
$p^{2k}$.
\end{corollary}

\begin{proof}
Since $L$ admits a 2-dimensional abelian ideal, 
the corollary follows from Proposition \ref{pabelianid}.
\end{proof}

\vspace{3mm}

In order to have a more elegant proof of simplicity in Lemma \ref{lssp} below,
we observe that the following generalization of 
\cite[Proposition 2.9.2]{NekSSgrp}
holds. Let $R$ be a PID, and let $K$ be the field of fractions of $R$.
We identify $R\subseteq K$. Let $d\in\bb{N}$, $\Phi:K^d\rar K^d$
be a $K$-linear function, and $p_\Phi(\lambda)\in K[\lambda]$
be the characteristic polynomial of $\Phi$. 
Let $M$ be the set of $x\in R^d$ such
that $\Phi(x)\in R^d$. Then $M$ is a sub-$R$-module of $R^d$
and the restriction $\varphi:M\rar R^d$ of $\Phi$ may be interpreted
as a virtual endomorphism of the $R$-module $R^d$ (in the application below,
$R^d$ is thought of as an abelian $R$-Lie lattice). 

\begin{proposition}\label{pabelsimp}
In the context described above, $D_\infty=\{0\}$ if and only if 
there are no monic irreducible factors of $p_\Phi(\lambda)$
with coefficients in $R$. 
\end{proposition}

\begin{proof}
The proof of \cite[Proposition 2.9.2]{NekSSgrp} works in this
more general context.
\end{proof}

\begin{lemma}\label{lssp}
Let $k\ges 1$ be an integer. Then the following Lie lattices are 
self-similar of index $p^k$. 
\begin{enumerate}
\item $L_{\ref*{C6}}(a)$. 
\item $L_{\ref*{C2}}(s,r,c)$ with $v_p(c)=1$.
\item $L_{\ref*{C7}}(s,a,c)$ with $v_p(c)=1$ and $v_p(a)\ges 1$,
or with $v_p(4c+a^2)=1$, $v_p(a)=0$ and $v_p(c)=0$.
\item For $p\ges 3$, $L_{\ref*{C7}}(s, 0, 1)$.
\end{enumerate}
\end{lemma}

\begin{proof} Let $(x_0,x_1,x_2)$ be the basis of the relevant Lie lattice
as given by its presentation in Definition \ref{dlielat}.
We begin with $L = L_{\ref*{C6}}(a)$, where we exhibit a simple virtual endomorphism
$\varphi:M\rar L$ of index $p^k$. Define 
$M= \langle x_0, x_1, p^kx_2\rangle $. 
For $a=0$, the abelian case, define 
$\varphi(x_0) = x_1$, 
$\varphi(x_1) = x_2$ and 
$\varphi(p^kx_2) = x_0$. 
For $a\neq 0$,
define 
$\varphi(x_0) = x_0$, 
$\varphi(x_1) = x_2$ and 
$\varphi(p^kx_2) = x_1$. 
Recall that $D_\infty$ is the intersection of the domains of the powers of $\varphi$.
In the abelian case one shows that $D_\infty =\{0\}$,
while in the non-abelian case one shows that $D_\infty=\gen{x_0}$.
Since a $\varphi$-invariant subset of $L$ has to be a subset 
of $D_\infty$, in both cases one  shows that a non-zero ideal 
of $L$ is not $\varphi$-invariant (cf. proof of Proposition \ref{pabelianid}).

We now denote by $L$ any of the Lie lattices that remain to be analyzed.
From Corollary \ref{csigmap2}, it is enough to treat the case where $k=2l+1$
is odd. We exhibit a simple virtual endomorphism
$\varphi:M\rar L$ of index $p^{2l+1}$.
For $L_{\ref*{C2}}(s,r,c)$, define 
$M= \langle x_0, p^lx_1, p^{l+1}x_2\rangle $
and 
$\varphi(x_0) = x_0$, 
$\varphi(p^l x_1) = x_1+p^{-1}cx_2$ and 
$\varphi(p^{l+1}x_2) = x_1+px_2$.
For $L_{\ref*{C7}}(s,a,c)$
with $v_p(c)=1$ and $v_p(a)\ges 1$, define
$M= \langle x_0, p^lx_1, p^{l+1}x_2\rangle $
and
$\varphi(x_0) = x_0$, 
$\varphi(p^l x_1) = p^{-1}cx_2$ and 
$\varphi(p^{l+1}x_2) = x_1-a x_2$.
For $L_{\ref*{C7}}(s,a,c)$ with $v_p(4c+a^2)=1$, $v_p(a)=0$ and $v_p(c)=0$
(necessarily $p\ges 3$), define
$M= \langle x_0, p^l(x_1-2^{-1}a x_2), p^{l+1}x_2\rangle$ and 
$\varphi(x_0) = x_0$, 
$\varphi(p^l (x_1-2^{-1}ax_2)) = p^{-1}(c+4^{-1}a^2)x_2$ and 
$\varphi(p^{l+1}x_2) = x_1-2^{-1}a x_2$.
Finally, for $L_{\ref*{C7}}(s,0,1)$, define 
$M= \langle x_0, p^l(x_1-x_2), p^{l+1}x_2\rangle$
and
$\varphi(x_0) = -x_0$, 
$\varphi(p^l (x_1-x_2) )= x_1+x_2$ and 
$\varphi(p^{l+1}x_2) = x_1-(1+p) x_2$.
The proof of simplicity of $\varphi$ may go as follows. 
Let $\psi: M\cap \gen{x_1,x_2}\rar \gen{x_1,x_2}$ be the restriction
of $\varphi$. Let $D_\infty$ be as above, and $E_\infty$ be the intersection
of the domains of the powers of $\psi$. We have 
$D_\infty = \gen{x_0}\oplus E_\infty$
(indeed, $\varphi$ is the direct sum of $\psi$ and a homomorphism
that sends $\gen{x_0}$ to $\gen{x_0}$). We claim that $E_\infty=\{0\}$,
from which the simplicty of $\varphi$ follows.
Observe that in each of the cases at hand $\psi$ is an isomorphism.
Because of that, one can see that the the virtual endomorphism associated
with $\Phi:=\psi\otimes \bb{Q}_p$ (as described above Proposition \ref{pabelsimp}) 
may be identified with $\psi$. Hence, by the proposition itself,
it sufficies to show that the characteristic polynomial 
$p(\lambda)\in \bb{Q}_p[\lambda]$ of 
$\Phi:\bb{Q}_px_1\oplus \bb{Q}_px_2\rar \bb{Q}_px_1\oplus\bb{Q}_px_2$
has no monic irreducible factors with coefficients in $\bb{Z}_p$.
We treat the case of $L_{\ref*{C2}}(s,r,c)$; the other cases are similar and
are left to the reader. We have
$\Phi(x_1) = p^{-l}x_1+p^{-l-1}cx_2$ 
and $\Phi(x_2)=p^{-l-1}x_1+p^{-l}x_2$.
Then 
$p(\lambda)=\lambda^2-2p^{-l}\lambda + p^{-2l}-cp^{-2l-2}$.
Observe that $v_p(p^{-2l}-cp^{-2l-2})=-2l-1<0$ so that in case $p(\lambda)$
is irreducible there is nothing left to prove.
Assume that $p(\lambda)$ is reducible.
The proof of the lemma is concluded
once we prove that this assumption leads to a contradiction. Indeed, 
$p(\lambda)=(\lambda - \mu)(\lambda -\nu)$ with
$\mu,\nu\in\bb{Q}_p$. We have $\mu+\nu = 2p^{-l}$ and 
$\mu\nu = p^{-2l}-cp^{-2l-2}$. 
Since $v_p(\mu)+v_p(\nu)=-2l-1$ then, without loss of generality, 
we can assume that $v_p(\mu)<v_p(\nu)$, so that $v_p(\mu)\les -l-1$.
It follows that $v_p(\mu+\nu) \les -l-1 <v_p(2p^{-l})=v_p(\mu+\nu)$, a contradiction.
\end{proof}


\subsection{Non-self-similarity results in dimension 3}\label{snssres3}

The main results of this relatively long technical section
are Propositions \ref{pll1nssp}, \ref{pl2nssp} and \ref{pl7nssp}.

\begin{remark}\label{rindpsubm2}
Let $L$ be a 3-dimensional $\bb{Z}_p$-lattice 
endowed with a basis $(x_0,x_1,x_2)$. 
For $e,f\in\bb{Z}_p$
we define index-$p$ submodules of $L$ by 
$L^{()}=\langle px_0,x_1,x_2 \rangle$,
$L^{(e)}=\langle x_0+ex_1,px_1,x_2\rangle$
and $L^{(e,f)}=\langle x_0+ex_2,x_1+fx_2,px_2\rangle$.
Any index-$p$ submodule of $L$ is isomorphic to $L^\xi$
for some $\xi = (), (e), (e,f)$. By changing $e$ or $f$ modulo $p$,
$L^{(e)}$ and $L^{(e,f)}$ do not change
(cf. \cite[Definition 2.22, Lemma 2.23]{NS2019}).
Observe that when $L^\xi$ is displayed as above, it is endowed with a basis.

Assume that $M$ is an index-$p$ submodule of $L$ endowed with a
basis $(y_0,y_1,y_2)$, and let $\varphi:M\rar L$ be a homomorphism of modules.
We denote by $\overline{F}=(F_{\alpha\beta})\in gl_3(\bb{Z}_p)$ 
the matrix of $\varphi$ relative to the respective
bases, namely, 
$\varphi(y_\beta) = \sum_{\alpha} F_{\alpha\beta}x_\alpha$
(cf. Section \ref{smet}).
\end{remark}

\vspace{5mm}

First, we treat the case where $\mr{dim}\,[L,L]=1$. 

\begin{lemma} \label{lll1}
Let $L$ be a 3-dimensional $\bb{Z}_p$-Lie lattice with $\mr{dim}[L,L]=1$.
Then the following holds.
\begin{enumerate}
\item \label{ldimcl}
$\mr{dim}\,Z(L)=1$.
\item \label{lcmm}
Let $M\subseteq L$ be a subalgebra of index $p$. Then
$Z(L) \subseteq M$ or $[M,M]=[L,L]$.
\end{enumerate}
\end{lemma}

\begin{proof}
There exist $s\in\bb{N}$, $r\in\bb{N}\cup\{\infty\}$
and a basis $(x_0,x_1,x_2)$ of $L$ such that 
$ [x_1,x_2]= 0$, 
$[x_0,x_1]= p^{s}(p^rx_1+x_2)$ and 
$[x_0,x_2]= 0$, where $p^\infty := 0$.
For item (\ref{ldimcl}), one easily checks that $Z(L)=\gen{x_2}$.
For item (\ref{lcmm}), one observes that 
if $M$ is of type $L^{()}$ or $L^{(e)}$ 
(cf. Remark \ref{rindpsubm2}) then $Z(L)\subseteq M$.
On the other hand, if $M$ is of type $L^{(e,f)}$, then it is 
a straightforward computation to show that $[M,M]=[L,L]$.
\end{proof}

\begin{lemma}\label{lsiminjll1}
Let $L$ be a 3-dimensional $\bb{Z}_p$-Lie lattice with $\mr{dim}[L,L]=1$.
Let $\varphi:M\rar L$ be a virtual endomorphism
of $L$.
If $\varphi$ is simple then $\varphi$ is injective.
\end{lemma}

\begin{proof} 
Assume that $\varphi$ is not injective.
We exhibit a non-trivial $\varphi$-invariant ideal $I$ of $L$.

Case 1: $\mr{ker}\,\varphi\subseteq Z(L)$. 
Since $\mr{dim}\,Z(L)=1$
then $\mr{dim}\,\mr{ker}\,\varphi =1$, so that there exists $k\in\bb{N}$
such that $p^kZ(L)\subseteq \mr{ker}\,\varphi$. Thus, it suffices to take $I=p^kZ(L)$.

Case 2: $\mr{ker}\,\varphi\not\subseteq Z(L)$.
There exists $z\in\mr{ker}\,\varphi$ such that $z\not\in Z(L)$,
so $[w,z]\neq 0$ for some
$w\in L$. 
Since $M$ has finite index in $L$, there exists $k\in\bb{N}$
such that $p^kw\in M$. Hence, $p^k[w,z]\neq 0$ and $p^k[w,z]\in\mr{ker}\,\varphi$.
By taking $x\in L$ such that $\mr{iso}_L[L,L]=\gen{x}$,
one sees that $p^k[w,z]=ax$ for some 
$a\in\bb{Z}_p$ with $a\neq 0$. Thus, it suffices to take $I=\gen{ax}$.
\end{proof}

\begin{proposition}\label{pll1nssp}
Let $L$ be a 3-dimensional $\bb{Z}_p$-Lie lattice with $\mr{dim}[L,L]=1$. 
Then $L$ is not self-similar of index $p$.
\end{proposition}

\begin{proof}
Let $\varphi: M\rar L$ be a virtual endomorphism of $L$
of index $p$. We prove that $\varphi$ is not simple by
either referring to a previous result or by exhibiting
a non-trivial $\varphi$-invariant ideal $I$ of $L$.
 If $[M,M]=[L,L]$ then it suffices to take $I=[L,L]$.
Otherwise, by item (\ref{lcmm}) of Lemma \ref{lll1},
we have $Z(L)\subseteq M$. Then $Z(L)=Z(M)$. Also, 
if $\varphi$ is not injective then $\varphi$ 
is not simple (Lemma \ref{lsiminjll1}); hence, we can assume that $\varphi$
is injective. Then $\mr{dim}\,\varphi(M)=\mr{dim}\,L$, so that
$\varphi(Z(M))\subseteq Z(L)$.
Thus, it suffices to take $I=Z(L)$.
\end{proof}

\vspace{5mm}

Next, we treat the case where $\mr{dim}\,[L,L]=2$.

\begin{lemma}\label{linvsubm}
In the context of Remark \ref{rindpsubm2}, 
assume that $F_{01}=F_{02}=0$. Then the following holds.
\begin{enumerate}
\item \label{linvsubm1}
Assume that $M= L^{()}$. Then 
$\gen{x_1,x_2}$ is $\varphi$-invariant.
\item \label{linvsubm2}
Assume that $M = L^{(e)}$, 
$p|F_{11}$ and $p|F_{21}$. Then
$\gen{px_1,px_2}$ is $\varphi$-invariant.
\item \label{linvsubm3}
Assume that $M=L^{(e,f)}$,  
$p|F_{12}$ and $p|F_{22}$. Then
$\gen{px_1,px_2}$ is $\varphi$-invariant.
\item \label{linvsubm4}
Assume that $M=L^{(e,f)}$, $f\equiv_p 0$, 
$p|F_{21}$ and $p|F_{22}$. Then
$\gen{x_1,px_2}$ is $\varphi$-invariant.
\end{enumerate}
\end{lemma}

\begin{proof}
We leave the simple proof to the reader.
\end{proof}

\begin{lemma}\label{lLe}
Let $L$ be a 3-dimensional $\bb{Z}_p$-Lie lattice with $\mr{dim}[L,L]=2$,
and $\bm{x}=(x_0,x_1,x_2)$ be a good basis of $L$. 
Let $M = \gen{x_0+ex_1,px_1,x_2}$ for some $e\in\bb{Z}_p$,
assume that $M$ is a subalgebra of $L$, and let $\varphi: M\rar L$ 
be a homomorphism of algebras. 
Let 
$$
A=p^{s}
\left[
\begin{array}{ll}
a & b\\
c & d
\end{array}
\right]
\qquad s\in\bb{N},\quad 
a,b,c,d\in\bb{Z}_p
$$
be the matrix of $L$ with respect to $\bm{x}$.
Moreover, assume that one of the following conditions is true:
\begin{enumerate}
\item $v_p(b)=0$; or
\item $a=d=1$, $v_p(b)\les v_p(c)$ and $b\neq 0$.
\end{enumerate}
Then $\gen{px_1,px_2}$ is a $\varphi$-invariant ideal of $L$.
\end{lemma}

\begin{proof}
Clearly, $I=\gen{px_1,px_2}$ is an ideal of $L$. Let $B$ be the matrix 
of $M$ with respect to the displayed basis, and $\ol{F}$ be the matrix of $\varphi$.
From item (\ref{cFeqs1}) of Corollary \ref{cFeqs} it follows that $FB=F_{00}AF$,
and this matrix equation is equivalent to the system of scalar equations
\begin{eqnarray}
a(1-F_{00})F_{11}+pcF_{12}-bF_{00}F_{21}&=&0 \label{eql1}\\[2pt]
-cF_{00}F_{11}+(a-dF_{00})F_{21}+pcF_{22}&=&0\label{eql2}\\[2pt]
bF_{11}+p(d-aF_{00})F_{12}-pbF_{00}F_{22}&=&0\label{eql3}\\[2pt]
-pcF_{00}F_{12} + bF_{21}+pd(1-F_{00})F_{22}&=&0\label{eql4}.
\end{eqnarray}
From item (\ref{linvsubm2}) of Lemma \ref{linvsubm} it is enough to show
that $p|F_{11}$ and $p|F_{21}$. Indeed, we claim  that $p|F_{11}$ and $p|F_{21}$,
and we proceed to prove the claim.
In case $v_p(b)=0$, the claim follows from
equations (\ref{eql3}) and (\ref{eql4}). Assume 
$a=d=1$ and $r:=v_p(b)\les v_p(c)$. If $v_p(1-F_{00})\ges r$ the claim
follows again from equations (\ref{eql3}) and (\ref{eql4}). 
In case $v_p(1-F_{00})<r$
the claim follows from the equations (\ref{eql1}) and (\ref{eql2}).
\end{proof}

\begin{proposition}\label{pl2nssp}
Let $s,r\in\bb{N}$ with $r\ges 1$, and $c\in\bb{Z}_p$ with 
$v_p(c)\neq 1$.
Then $L_2(s,r,c)$ is not self-similar of index $p$.
\end{proposition}

\begin{proof}
Observe that $\mr{dim}[L,L]=2$.
Let $\varphi :M\rar L$ be a virtual endomorphism of $L$ of index $p$.
We will show that there exists a non-trivial $\varphi$-invariant ideal $I$ 
of $L$, from which the proposition follows.
Observe that $\gen{x_1,x_2}$, $\gen{px_1,px_2}$ and $\gen{x_1,px_2}$
are non-trivial ideals of $L$. The $\varphi$-invariance of the various 
$I$ defined below follows from Lemma \ref{lmetmor}(\ref{lmetmor2}) 
and Lemma \ref{linvsubm}.
Observe that the matrix equation $FB=F_{00}AF$ of item (\ref{cFeqs1}) of 
Corollary \ref{cFeqs} holds.
We divide the proof in four cases.

Case 1: $M=L^{()}$. It suffices to take $I = \gen{x_1,x_2}$.
Case 2: $M=L^{(e)}$. It suffices to take $I = \gen{px_1,px_2}$ (Lemma \ref{lLe}).
Case 3: $M = L^{(e,f)}$ with $f\not\equiv_p 0$.
The matrix equation $FB=F_{00}AF$ implies that the following equations
hold true:
$$
\left\{
\begin{array}{l}
p(1+p^rf-F_{00})F_{11}+(-f-p^rf^2+p^rc+p^2f)F_{12}-p^{r+1}F_{00}F_{21}=0\\[2pt]
-p^{r+1}cF_{00}F_{11}+p(1+p^rf-F_{00})F_{21}+(-f-p^rf^2+p^rc+p^2f)F_{22}=0,
\end{array}
\right.
$$
from which we can see that $p|F_{12}$ and $p|F_{22}$. 
Thus, it suffices to take $I=\gen{px_1,px_2}$.
Case 4: $M = L^{(e,0)}$.
The matrix equation $FB=F_{00}AF$ is equivalent to the equations
\begin{eqnarray}
(1-F_{00})F_{11}+p^{r-1}cF_{12}-p^rF_{00}F_{21}&=&0\label{eql21}\\[2pt]
-p^rcF_{00}F_{11}+(1-F_{00})F_{21}+p^{r-1}cF_{22}&=&0\label{eql22}\\[2pt]
p^{r+1}F_{11}+(1-F_{00})F_{12}-p^rF_{00}F_{22}&=&0 \label{eql23}\\[2pt]
-p^rcF_{00}F_{12}+p^{r+1}F_{21}+(1-F_{00})F_{22}&=&0.\label{eql24}
\end{eqnarray}
If $v_p(1-F_{00})<r$ then equations (\ref{eql23}) and (\ref{eql24}) 
imply that $p|F_{12}$ and $p|F_{22}$, and we can take 
$I=\gen{px_1, px_2}$.
Assume $l:=v_p(1-F_{00})\ges r$. Observe that, since $r\ges 1$, $F_{00}\in\bb{Z}_p^*$.
We divide the proof into two cases,
according whether $v_p(c)\ges 2$ or $v_p(c)=0$.
\begin{enumerate}
\item Assume $v_p(c)\ges 2$.
 \begin{enumerate}
 \item Assume $l\ges r+1$. 
 From equation (\ref{eql23}), we have $p|F_{22}$, 
 so that $p|F_{21}$ (equation (\ref{eql24})). Thus, it suffices to take
 $I=\gen{x_1,px_2}$.
 \item Assume $l=r$. 
 From equation (\ref{eql24}), we have $p|F_{22}$, 
 so that $p|F_{12}$ (equation (\ref{eql23})). 
 Thus, it suffices to take $I=\gen{px_1,px_2}$.
 \end{enumerate}
\item Assume $v_p(c)=0$. 
 From equation (\ref{eql21}), we have $p|F_{12}$; 
 from equation (\ref{eql22}), we have $p|F_{22}$. 
 Thus, it suffices to take $I=\gen{px_1,px_2}$.
\end{enumerate}
\end{proof}

\begin{lemma}\label{lC7Feq}
Let $s\in\bb{N}$ and $a,c,e,f\in\bb{Z}_p$ with $c\neq 0$. 
Define $L=L_{\ref*{C7}}(s,a,c)$, where $L$ 
is endowed with the basis $(x_0,x_1,x_2)$ given in Definition
\ref{dlielat}.
Let $M=\gen{x_0+ex_2, x_1+fx_2,px_2}$ and assume that $M$ is a subalgebra of $L$.
Let $\varphi:M\rar L$ be homomorphism of algebras,
and $F$ be the matrix of $\varphi$ with respect to the given bases.
Then
\begin{eqnarray}
pF_{21} - F_{00}[pfF_{11}+(c-af-f^2)F_{12}] & = & 0\label{eq1}\\[3pt]
F_{22} - F_{00}[pF_{11} - (a+f)F_{12}] & = & 0\label{eq2} \\[3pt]
(F_{00}-1)[-p(1+F_{00})F_{11}+(f(1+F_{00})+aF_{00})F_{12}] &=& 0\label{eq3}\\[3pt]
(F_{00}-1)[p(a+f(1+F_{00}))F_{11}+(1+F_{00})(c-af-f^2)F_{12}]&=&0\label{eq4}.
\end{eqnarray}
\end{lemma}

\begin{proof}
The result follows from Corollary \ref{cFeqs}(\ref{cFeqs1}).
\end{proof}

\begin{lemma}\label{lC7thenI}
In the context of Lemma \ref{lC7Feq},
the following holds.
\begin{enumerate}
\item \label{lC7thenI0}
Assume $p|F_{12}$. Then $\gen{px_1,px_2}$ is a $\varphi$-invariant ideal of $L$.
\item \label{lC7thenI3}
Assume $c-af-f^2=0$, $a\neq 0$, $2f+a\not\equiv_p 0$, and
$F_{12}\in\bb{Z}_p^*$. Then $\gen{x_1+fx_2}$ 
is a $\varphi$-invariant ideal of $L$.
\item \label{lC7thenI4}
Assume $p\ges 3$, $f=-2^{-1}a$, $v_p(a)=0$, $v_p(4c+a^2)\ges 2$ and $F_{12}\neq 0$.
Then $\gen{x_1-2^{-1}ax_2, px_2}$ is a $\varphi$-invariant ideal of $L$.
\end{enumerate}
\end{lemma}

\begin{proof}
\begin{enumerate}
\item 
From equation (\ref{eq2}) of Lemma \ref{lC7Feq}
it follows that $p|F_{22}$. Now the item 
follows from
item (\ref{linvsubm3}) of Lemma \ref{linvsubm}.
\item
Observe that $f\neq 0,-a$, since $c\neq 0$.
One checks directly that $[x_0,x_1+fx_2]$ is a 
$\bb{Z}_p$-multiple of $x_1+fx_2$, hence, $I=\gen{x_1+fx_2}$ is an ideal of $L$.
We have $\varphi(x_1+fx_2)=F_{11}(x_1+fx_2)+(F_{21}-fF_{11})x_2$.
 \begin{enumerate}
 \item 
Case $F_{00}=1$. 
From equation (\ref{eq1}) of Lemma \ref{lC7Feq}
we have $F_{21}=fF_{11}$, hence, $I$ is $\varphi$-invariant.
\item
Case $F_{00}\neq 1$. Since $F_{12}\in\bb{Z}_p^*$, 
equations (\ref{eq3}) and (\ref{eq4})
of Lemma \ref{lC7Feq} have a non-trivial solution in the variables
$F_{11}$ and $F_{12}$. It follows that 
$[f(1+F_{00})+aF_{00}][a+f(1+F_{00})]=0$.
  \begin{enumerate}
  \item[(i)]
Case $f(1+F_{00})+aF_{00}=0$. Hence, 
$$
F_{00} = -\frac{f}{a+f}, \qquad\qquad
a+f(1+F_{00})=a\frac{a+2f}{a+f}\neq 0. 
$$
Hence $F_{11} = 0$ (equation (\ref{eq4})), so that $F_{21}=0$
(equation (\ref{eq1})). Hence, $I$ is $\varphi$-invariant.
  \item[(ii)]
Case $a+f(1+F_{00})=0$. We show that we have a contradiction.
Indeed, $F_{00} = -\frac{a+f}{f}$, so that
$a[pF_{11}-(2f+a)F_{12}] = 0$ (equation (\ref{eq3})), 
and consequently $pF_{11}=(2f+a)F_{12}\in\bb{Z}_p^*$, 
which is a contradiction.
  \end{enumerate}
 \end{enumerate}
\item
By applying $[x_0,\cdot\,]$ to its generators, one sees that
$I=\gen{x_1-2^{-1}ax_2, px_2}$ is an ideal of $L$.
We have
(cf. item (\ref{lmetmor2}) of Lemma \ref{lmetmor})
$$
\left\{
\begin{array}{l}
\varphi(x_1-2^{-1}ax_2) = F_{11}(x_1-2^{-1}ax_2)+ (F_{21}+2^{-1}aF_{11})x_2\\[3pt]
\varphi(px_2) = F_{12}(x_1-2^{-1}ax_2)+ (F_{22}+2^{-1}aF_{12})x_2,
\end{array}
\right.
$$
from which we see that to show that
$I$ is $\varphi$-invariant is  equivalent to show
that $p|(F_{21}+2^{-1}aF_{11})$ and $p|(F_{22}+2^{-1}aF_{12})$.
We claim that, indeed, $p|(F_{21}+2^{-1}aF_{11})$ and $p|(F_{22}+2^{-1}aF_{12})$.
In fact, equations (\ref{eq1}) and (\ref{eq2}) of Lemma \ref{lC7Feq} are equivalent to
$$
\left\{
\begin{array}{l}
F_{21}+2^{-1}aF_{11} =
-(F_{00}-1)2^{-1}aF_{11}+F_{00}p^{-1}(c+4^{-1}a^2)F_{12}\\[3pt]
F_{22}+2^{-1}aF_{12} = pF_{00}F_{11} -(F_{00}-1)2^{-1}aF_{12}.
\end{array}
\right.
$$
 \begin{enumerate}
 \item
Case $F_{00}=1$. The claim is obviously true. 
 \item 
Case $F_{00}\neq 1$. 
Since $F_{12}\neq 0$,
equations (\ref{eq3}) and (\ref{eq4}) of Lemma \ref{lC7Feq} have a non-trivial solution in the variables
$F_{11}$ and $F_{12}$. Hence,
the determinant of the coefficient matrix has to be zero, which implies that
$$ a^2(F_{00}-1)^2 - (F_{00}+1)^2(4c+a^2) = 0.$$
It follows that $p|(F_{00}-1)$, so the claim is true.
 \end{enumerate}
\end{enumerate}
\end{proof}

\begin{corollary}\label{csimroot}
In the context of Lemma \ref{lC7Feq}, 
assume that $a\neq 0$ and that $f$ is a simple root modulo $p$ of the polynomial 
$P(\kappa)=\kappa^2+a\kappa-c$.
Then $\varphi$ is not simple.
\end{corollary}

\begin{proof}
From Hensel's Lemma it follows that there exists $\bar{f}\in\bb{Z}_p$ such that 
$\bar{f}\equiv_p f$ and $P(\bar{f})=0$. Clearly, $2\bar{f}+a\not\equiv_p 0$ and 
$M =\gen{x_0+ex_2, x_1+\bar{f}x_2, px_2}$. In other words, 
we can assume $c-af-f^2=0$ and $2f+a\not\equiv_p 0$.
In case $p|F_{12}$, $\gen{px_1,px_2}$ is a non-trivial $\varphi$-invariant ideal of $L$
by Lemma \ref{lC7thenI}(\ref{lC7thenI0}). 
In case $p\nmid F_{12}$,
$\gen{x_1+fx_2}$ 
is a non-trivial $\varphi$-invariant ideal of $L$ by
Lemma \ref{lC7thenI}(\ref{lC7thenI3}). Hence, $\varphi$ is not simple.
\end{proof}

\begin{proposition}\label{pl7nssp}
Let $s\in\bb{N}$ and $a,c\in\bb{Z}_p$ with $c\neq 0$.
Assume that one of the following conditions is satisfied:
\begin{enumerate}
\item \label{pl71} $v_p(a)\ges 1$ and $v_p(c)\ges 2$.
\item \label{pl72} $v_p(a) =0$ and $v_p(c)\ges 1$.
\item \label{pl73} $p\ges 3$, $a\neq 0$, $v_p(a)\ges 1$ and $v_p(c) = 0$.
\item \label{pl74} $a=0$, $v_p(c)  =0$ and $c$ is not a square modulo $p$.
\item \label{pl75} $v_p(a) = 0$, $v_p(c)=0$ and $v_p(4c+a^2)\neq 1$. 
\end{enumerate}
Then $L_{\ref*{C7}}(s,a,c)$ is not self-similar of index $p$. 
\end{proposition}

\begin{proof}
Observe that $\mr{dim}[L,L]=2$. 
Denote $L=L_{\ref*{C7}}(s,a,c)$, and let 
$\varphi :M\rar L$ be a virtual endomorphism of $L$ of index $p$.
We will show that $\varphi$ is not simple by either applying a previously proven result,
or by exhibiting a $\varphi$-invariant ideal $I$ 
of $L$. Recall Remark \ref{rindpsubm2}.
If $M = \gen{px_0, x_1, x_2}$ then it suffices to take $I= \gen{x_1,x_2}$ 
(Corollary \ref{cFeqs}(\ref{cFeqs2})).
If $M = \gen{x_0+ex_1, px_1, x_2}$, where $e\in\bb{Z}_p$, 
then it suffices to take $I= \gen{px_1,px_2}$ (Lemma \ref{lLe}). Assume 
$M = \gen{x_0+ex_2, x_1+fx_2, px_2}$, where $e,f\in\bb{Z}_p$
(the last case to be treated). 
By Lemma \ref{lC7thenI}(\ref{lC7thenI0}),
we can assume $F_{12}\in\bb{Z}_p^*$. We observe that this implies that
$c-af-f^2\equiv_p 0$. We divide the proof in several cases,
depending on which assumption of the statement
holds.
\begin{enumerate}
\item 
Assume $v_p(a)\ges 1$ and $v_p(c)\ges 2$. 
Hence, $f\equiv_p 0$,
and it follows that $p|F_{21}$ and $p|F_{22}$. Thus, it suffices to take $I =\gen{x_1, px_2}$,
which is an ideal of $L$ (since, in particular, $v_p(c)\ges 1$)
and is $\varphi$-invariant by Lemma \ref{linvsubm}(\ref{linvsubm4}).
\item
Assume that $v_p(a) =0$ and $v_p(c)\ges 1$, or that 
$p\ges 3$, $a\neq 0$, $v_p(a)\ges 1$ and $v_p(c) = 0$.
Then $f$ is a simple root
of the polynomial $P(\kappa)=\kappa^2 + a\kappa -c$ modulo $p$.
Applying Corollary \ref{csimroot}, we see that $\varphi$ is not simple.
\item
Assume $a=0$, $v_p(c)  =0$ and $c$ is not a square modulo $p$.
This case contradicts $c-af-f^2\equiv_p 0$.
\item
Assume  $v_p(a) = 0$, $v_p(c)=0$ and $v_p(4c+a^2)\neq 1$. 
Case 1: $v_p(4c+a^2) = 0$.
Then $f$ is simple root
of the polynomial $P(\kappa)=\kappa^2 + a\kappa -c$ modulo $p$.
Applying Corollary \ref{csimroot}, we see that $\varphi$ is not simple.
Case 2: $v_p(4c+a^2) \ges 2$. Then $p\ges 3$, 
and $f\equiv_p -2^{-1}a$. We can assume $f=-2^{-1}a$.
From Lemma \ref{lC7thenI}(\ref{lC7thenI4}), we can take 
$I=\gen{x_1-2^{-1}ax_2, px_2}$.
\end{enumerate}
\end{proof}


\subsection{Self-similarity of 3-dimensional solvable Lie lattices}\label{smain3}

\begin{remark}\label{rclassif}
Assume $p\ges 3$.
Any 3-dimensional solvable $\bb{Z}_p$-Lie lattice is isomorphic to exactly one 
of the Lie lattices in the list below (see Definition \ref{dlielat} for 
the notation and \cite[Proposition 7.3]{GSKpsdimJGT} for the proof).
We also give necessary and sufficient conditions for the 
respective Lie lattice to be residually nilpotent (cf. \cite[page 731]{GSKpsdimJGT}).
For $p\ges 5$ the residually nilpotent Lie lattices in the list
provide a classification of 3-dimensional solvable torsion-free $p$-adic analytic
pro-$p$ groups (cf. \cite[Theorem B]{GSKpsdimJGT}).

\begin{enumerate}
\setcounter{enumi}{-1}
\item $L_{\ref*{C0}}$. It is residually nilpotent (abelian).
\item $L_{\ref*{C1}}(s)$. It is residually nilpotent if and only if $s\ges 1$.
\item $L_{\ref*{C2}}(s,r,c)$ with $r\ges 1$.
It is residually nilpotent if and only if $s\ges 1$.
\item $L_{\ref*{C3}}(s)$. It is residually nilpotent (nilpotent).
\item $L_{\ref*{C4}}(s,t,\vep)$.
It is residually nilpotent if and only if $s\ges 1$ or $t \ges 1$.
\item $L_{\ref*{C5}}(s,r,c)$.
It is residually nilpotent if and only if $s\ges 1$ holds, or 
$r\ges 1$ and $v_p(c)\ges 1$ hold.
\end{enumerate}
\end{remark}

Recall that the \textbf{self-similarity index} of 
a self-similar $\bb{Z}_p$-Lie lattice $L$ is the smallest power
of $p$, say $p^k$, such that $L$ is self-similar of index $p^k$.

\begin{theorem}\label{talgpodd}
Assume $p\ges 3$.
Let $L$ be a 3-dimensional solvable $\bb{Z}_p$-Lie lattice, and 
let $\sigma$ be the self-similarity index of $L$. Then $\sigma = p$ or
$\sigma = p^2$. Moreover, $\sigma =p$ if and only if $L$ is isomorphic 
to one of the Lie lattices that appear in the following sublist
of the list given in Remark \ref{rclassif}.
\begin{enumerate}
\setcounter{enumi}{-1} 
\item $L_{\ref*{C0}}$. 
\item $L_{\ref*{C1}}(s)$.
\item $L_{\ref*{C2}}(s,r,c)$  with $v_p(c)=1$ (and $r\ges 1$).
\refstepcounter{enumi}
\item $L_{\ref*{C4}}(s,t,\vep)$ with $t=1$, or with $t=0$ and $\vep =0$. 
\item $L_{\ref*{C5}}(s,r,c)$ with $r\ges 1$ and $v_p(c)=1$, or with 
$r=0$ and $v_p(4c+1) = 1$.
\end{enumerate}
\end{theorem}

\begin{proof}
By Corollary \ref{csigmap2}, $\sigma = p$ or $\sigma=p^2$. 
Observe that $L_{\ref*{C0}}=L_{\ref*{C6}}(0)$, 
$L_{\ref*{C1}}(s)=L_{\ref*{C6}}(p^s)$,
$L_{\ref*{C4}}(s,t,\vep)=L_{\ref*{C7}}(s,0,p^t\rho^\vep)$ and
$L_{\ref*{C5}}(s,r,c)=L_{\ref*{C7}}(s, p^r, c)$.
The claim that the Lie lattices in the 
statement are self-similar of index $p$ follows from Lemma \ref{lssp}.
The remaining Lie lattices of Remark \ref{rclassif} (the ones not
in the statement) are not self-similar of index $p$
by Propositions \ref{pll1nssp}, \ref{pl2nssp} and \ref{pl7nssp}.
\end{proof}


\subsection{Non-self-similarity results in higher dimension}\label{snssresh}

The main results of this section are Proposition \ref{phdn} and 
Corollary \ref{cnsspsub}; the latter 
is a key ingredient in the proof of Theorem \ref{tshssp}.

Let $d\ges 2$ be an integer.
As in Section \ref{smet}, Greek indices will take values
in $\{0,...,d\}$, while Latin indices will take values in $\{1,...,d\}$.
We denote the $p$-adic valuation 
by $v$ instead of $v_p$.

\begin{definition}\label{dLab}
Let $a=(a_1,...,a_d)\in\bb{Z}_p^d$ and $b=(b_1,...,b_{d-1})\in\bb{Z}_p^{d-1}$.
We define an antisymmetric $(d+1)$-dimensional $\bb{Z}_p$-algebra $L(a,b)$
as follows. As a $\bb{Z}_p$-module, $L=\bb{Z}_p^{d+1}$. Denoting by
$(x_0,...,x_d)$ the canonical basis of $L$,  
the bracket of $L(a,b)$ is induced by
the commutation relations
$$
\left\{
\begin{array}{ll}
[x_i,x_j]=0&\\[2pt]
[x_0,x_1]=\sum_i a_ix_i&\\[2pt]
[x_0,x_{i+1}]= b_ix_i&\mbox{ if }\,\,i<d.
\end{array}
\right.
$$
\end{definition}

\begin{remark}
$L(a,b)$ is a metabelian (possibly abelian) Lie lattice.
\end{remark}

We will prove the following proposition at the end of the section.

\begin{proposition}\label{phdn}
Let $a\in\bb{Z}_p^d$ and $b\in\bb{Z}_p^{d-1}$ be as in Definition \ref{dLab}. 
Assume that:
\begin{enumerate}
\item \label{phdn1}
$a_d\neq 0$.
\item \label{phdn2}
$v(b_i)<v(b_{i+1})$ whenever $i<d-1$.
\item \label{phdn3}
$v(b_i)<v(a_i)$ whenever $i<d$.
\item \label{phdn4}
$v(b_{d-1})+1 <v(a_d)$.
\end{enumerate}
Then $L(a,b)$ is not self similar of index $p$.
\end{proposition}

\vspace{0mm}

\begin{corollary}\label{cnsspsub}
Let $a\in\bb{Z}_p^d$ and $b\in\bb{Z}_p^{d-1}$ be as in Definition \ref{dLab}. 
Assume that $a_d\neq 0$ and $b_1=...=b_{d-1}=1$.
Then $L(a,b)$ admits a finite-index subalgebra that is not self-similar of index $p$.
\end{corollary}

\begin{proof}
Let $L=L(a,b)$ and
take $k_0,...,k_d\in\bb{N}$ as follows. Choose 
$$
k_0 > \frac{d-1}{2},\qquad\qquad
k_1\ges \mr{max}\left( (i-1)k_0-\frac{(i-1)(i-2)}{2}\right)_{i=1,...,d}
$$
and, for $i\ges 2$, define 
$$
k_i = k_1-(i-1)k_0+\frac{(i-1)(i-2)}{2}.
$$
It is not difficult to show that $k_0+k_1-k_i>i-1$ for all $i$.
Define $M=\gen{p^{k_0}x_0,...,p^{k_d}x_d}$. Then $M$ is a finite-index subalgebra
of $L$ which is isomorphic to $L(a',b')$, where
$$
\left\{
\begin{array}{ll}
b'_i = p^{i-1}&\mbox{ if }\,\,i<d\\[2pt]
a'_i = p^{k_0+k_1-k_i}a_i.&
\end{array}
\right.
$$
By Proposition \ref{phdn}, $M$ is not self-similar of index $p$. 
\end{proof}

\vspace{5mm}

The remainder of the section is devoted to the proof 
of Proposition \ref{phdn}.

\begin{remark}\label{rlabid}
Let $a\in\bb{Z}_p^d$, $b\in\bb{Z}_p^{d-1}$ and $L=L(a,b)$ (Definition \ref{dLab}). 
We define $I_0=\gen{x_1,...,x_d}$ and 
$I_i=\gen{x_1,...,x_{i-1},px_i,...,px_d}$. Hence,
$I_1=pI_0$ and $I_1\subset I_2\subset ...\subset I_d\subset I_0$.
Moreover, $I_0$ and $I_1$ are ideals of $L$.
\end{remark}

\begin{lemma}\label{llabid}
Let $i>1$ and $I_i$ be defined as in Remark \ref{rlabid}.
Then $I_i$ is an ideal of $L$ if and only if
$p|a_j$ for all $j\ges i$.
\end{lemma}

\begin{proof}
It suffices to 
observe that $I_i$ is an ideal of $L$ if and only if $[x_0,y]\in I_i$ for all 
the generators $y$ of $I_i$ displayed in the definition of $I_i$. 
\end{proof}

\begin{lemma}\label{lpFij}
Let $a_1,...,a_d,b_1,...,b_{d-1}\in\bb{Z}_p$ and define $A=(A_{ij})\in gl_d(\bb{Z}_p)$
by 
$$
A_{ij} = 
\left\{
\begin{array}{ll}
a_i&\mbox{ if }\,\,j=1\\[2pt]
b_i&\mbox{ if }\,\,j=i+1\\[2pt]
0&\mbox{ if }\,\,j>1\mbox{ and }j\neq i+1.
\end{array}
\right.
$$
Let $i_0\in\{1,...,d\}$ and $f_1,...,f_{i_0-1}\in\bb{Z}_p$ 
(no choice of coefficients `$f$' has to be made when $i_0=1$). 
Define $U=(U_{ij})\in gl_{d}(\bb{Z}_p)$ by
$$
U_{ij} = 
\left\{
\begin{array}{rl}
p&\mbox{ if }\,\,i=j=i_0\\[2pt]
1&\mbox{ if }\,\, i=j\neq i_0\\[2pt]
-f_j&\mbox{ if }\,\, i>j \mbox{ and }i=i_0\\[2pt]
0&\mbox{ if }\,\, i>j \mbox{ and } i\neq i_0\\[2pt]
0&\mbox{ if }\,\, i<j.
\end{array}
\right.
$$
Let $\widehat{U}$ be the cofactor matrix of $U$. 
Let $F_{00}\in\bb{Z}_p$
and $F=(F_{ij})\in gl_d(\bb{Z}_p)$. 
Assume that
$F\widehat{U}^TA= F_{00}A F\widehat{U}^T$,
and that the $a_i$'s and $b_j$'s satisfy the four assumptions in
the statement of Proposition \ref{phdn}.
Then the following holds.
\begin{enumerate}
\item \label{lpFij2} 
Assume that $f_k= 0$ for all $k<i_0$. Then
$p|F_{ij}$ for $i\ges i_0$ and $j\les i_0$.
\item \label{lpFij1}
Assume that there exists $k_0<i_0$ such that $f_{k_0}\not\equiv_p 0$
and $f_k = 0$ for all $k<k_0$. 
Then
$p|F_{i,i_0}$ for $i\ges k_0$, and $p|F_{ij}$ for $i\ges k_0$ and $j<k_0$.
\end{enumerate}
\end{lemma}

\begin{proof}
We have
$$
\widehat{U}^T_{ij} = 
\left\{
\begin{array}{rl}
1&\mbox{ if }\,\, i=j=i_0\\[2pt]
p&\mbox{ if }\,\,i=j\neq i_0\\[2pt]
f_j&\mbox{ if }\,\,i>j\mbox{ and } i=i_0\\[2pt]
0&\mbox{ if }\,\,i>j\mbox{ and } i\neq i_0\\[2pt]
0&\mbox{ if }\,\,i<j.
\end{array}
\right.
$$
A straightforward computation gives
$$(F\widehat{U}^TA)_{ik}=
\left\{
\begin{array}{ll}
\sum_{j<i_0}(pa_jF_{ij}+f_ja_jF_{i,i_0})+
a_{i_0}F_{i,i_0}+\sum_{j>i_0}pa_j F_{ij}&
\mbox{ if }\,\,k=1\\[2pt]
pb_{k-1}F_{i,k-1}+f_{k-1}b_{k-1}F_{i,i_0}&
\mbox{ if }\,\,1<k<i_0+1\\[2pt]
b_{i_0}F_{i,i_0}&
\mbox{ if }\,\,k=i_0+1\\[2pt]
pb_{k-1}F_{i,k-1}&
\mbox{ if }\,\,k>i_0+1
\end{array}
\right.
$$
and
\small
$$
(F_{00}A F\widehat{U}^T)_{ik}=
\left\{
\begin{array}{ll}
F_{00}pa_iF_{1k}+F_{00}f_ka_iF_{1,i_0}+
F_{00}pb_i F_{i+1,k}+F_{00}f_kb_iF_{i+1,i_0}&
\mbox{ if }\,\,i<d\mbox{ and }k<i_0\\[2pt]
F_{00}a_iF_{1,i_0}+F_{00}b_iF_{i+1,i_0}&
\mbox{ if }\,\,i<d\mbox{ and } k=i_0\\[2pt]
F_{00}pa_iF_{1k}+F_{00}pb_iF_{i+1,k}&
\mbox{ if }\,\,i<d\mbox{ and } k>i_0\\[2pt]
F_{00}pa_dF_{1k}+F_{00}f_ka_dF_{1,i_0}&
\mbox{ if }\,\,i=d\mbox{ and } k<i_0\\[2pt]
F_{00}a_dF_{1,i_0}&
\mbox{ if }\,\,i=d\mbox{ and } k=i_0\\[2pt]
F_{00}pa_dF_{1k}&
\mbox{ if }\,\,i=d\mbox{ and } k>i_0.
\end{array}
\right.
$$
\normalsize
We denote by $F(i,k)$ the equality 
$(F\widehat{U}^TA)_{ik}= (F_{00}A F\widehat{U}^T)_{ik}$,
which is true for all $i$ and $k$ by assumption. 
Observe that $b_i\neq 0$ whenever $i<d$.
We divide the proof of the two items of the statement of the theorem in four cases.
The fourth case will be treated in detail, while the
details of the other cases are left to the reader.

\begin{enumerate}
\item \textit{Item (\ref{lpFij1}) of the statement, proof of $p|F_{i,i_0}$.}
The claim $p|F_{i,i_0}$ follows from $F(i,k_0+1)$. 
The proof has to be done by descending induction on $i$,
since for $i=k_0$
one needs to use that $p|F_{k_0+1,i_0}$.
\item 
\textit{Item (\ref{lpFij1}) of the statement, proof of $p|F_{ij}$.}
The claim $p|F_{ij}$
follows from $F(i,j+1)$.
We observe that for $i=k_0$ and $j=k_0-1$ one has also to use that
$p|F_{k_0+1,i_0}$, which was proven in previous item.
\item 
\textit{Item (\ref{lpFij2}) of the statement, case $i_0<d$.}
The claim $p|F_{ij}$
follows from $F(i,j+1)$.
The proof has to be done by descending induction on $j$,
since for $i=i_0$ and $j=i_0-1$
one needs to use that $p|F_{i_0+1,i_0}$.
\item 
\textit{Item (\ref{lpFij2}) of the statement, case $i_0=d$.}
We have to prove that $p|F_{dj}$ for all $j$.
\begin{enumerate}
\item Assume $j<d-1$. The equation $F(d,j+1)$ reads
$pb_jF_{dj}=F_{00}pa_dF_{1,j+1}$.
Hence, $v(F_{dj})\ges v(a_d)-v(b_j)$. In particular, $p|F_{dj}$.
\item Assume $j=d-1$. The equation $F(d,j+1)$ reads
$pb_{d-1}F_{d,d-1}=F_{00}a_dF_{1,d}$.
Hence, $v(F_{d,d-1})\ges v(a_d)-v(b_{d-1})-1$. In particular, $p|F_{d,d-1}$.
\item Assume $j=d$. The equation $F(d,1)$ reads
$ \sum_{j<d}pa_jF_{dj} + a_d F_{dd}=F_{00}pa_dF_{11}$.
For all $j<d-1$, we have 
$v(pa_jF_{dj})=1+v(a_j)+v(F_{dj})\ges 1+v(a_j)+ v(a_d)-v(b_j)>v(a_d)$.
For $j=d-1$, we have 
$v(pa_{d-1}F_{d,d-1})=1+v(a_{d-1})+v(F_{d,d-1})\ges 
1+v(a_{d-1})+ v(a_d)-v(b_{d-1})-1>v(a_d)$. Hence, $p|F_{dd}$.
\end{enumerate}

\end{enumerate}
\end{proof}

\begin{remark}\label{rmbases}
Let $L$ be a $(d+1)$-dimensional $\bb{Z}_p$-lattice endowed with a basis
$(x_0,...,x_d)$, and let $M\subseteq L$ be an index-$p$ submodule.
Exactly one of the following cases holds (cf. \cite[Lemma 2.23]{NS2019}):
\begin{enumerate}
\item \label{ihdpa1}
$(y_0,...,y_d)$ is a basis of $M$, where
$y_0 = px_0$ and $y_i = x_i$.
\item \label{ihdpa2}
There exist $i_0\in\{1,...,d\}$ and $f_0\in\bb{Z}_p$
such that $(y_0,...,y_d)$ is a basis of $M$, where
$y_0=x_0-f_0x_{i_0}$ and 
$$
y_i =
\left\{
\begin{array}{ll}
x_i  & \mbox{ if }\,\,i\neq i_0\\
px_{i_0} & \mbox{ if }\,\,i= i_0.
\end{array}
\right.
$$
\item \label{ihdpa3}
There exist $k_0,i_0\in\{1,...,d\}$ 
and $f_0, f_{k_0},...,f_{i_0-1}\in\bb{Z}_p$
such that $k_0<i_0$, $f_{k_0}$ is invertible in $\bb{Z}_p$,
and $(y_0,...,y_d)$ is a basis of $M$, where
$y_0=x_0-f_0x_{i_0}$ and 
$$
y_i =
\left\{
\begin{array}{ll}
x_i  & \mbox{ if }\,\,i < k_0\\
x_i-f_ix_{i_0} & \mbox{ if }\,\,k_0\les i < i_0\\
px_{i_0} & \mbox{ if }\,\, i = i_0\\
x_{i} & \mbox{ if }\,\, i > i_0.
\end{array}
\right.
$$

\end{enumerate}

\end{remark}

\begin{lemma}\label{llabinv} 
Let $L$ be a $(d+1)$-dimensional lattice, let $M\subseteq L$ be an index-$p$
submodule, and let $\varphi:M\rar L$ be a homomorphism of modules.
Let $(x_0,...,x_d)$ be a basis of $L$, 
let $(y_0,...,y_d)$ be a basis of $M$,
and let $y_\beta = \sum_\alpha F_{\alpha\beta}x_\alpha$.
Assume $F_{0i}=0$ for all $i$.
Let $I_i:= \gen{z_1,...,z_d}$, where 
$$
z_j =
\left\{
\begin{array}{ll}
x_j  & \mbox{ if }\,\,j<i\\
px_j & \mbox{ if }\,\,j\ges i
\end{array}
\right.
$$
(cf. Remark \ref{rlabid}).
Then the following holds.
\begin{enumerate}
\item \label{llabinv2} 
Assume that $(y_0,...,y_d)$ has the form displayed in case 
(\ref{ihdpa2}) of Remark \ref{rmbases}. Then:
\begin{enumerate}
 \item
 $I_{i_0}\subseteq M$.
 \item $\varphi(I_{i_0})\subseteq I_{i_0}$ if and only if 
 $p|F_{ij}$ for $i\ges i_0$ and $j\les i_0$.
 \end{enumerate}
\item \label{llabinv1}
Assume that $(y_0,...,y_d)$ has the form displayed in case 
(\ref{ihdpa3}) of Remark \ref{rmbases}. Then:
\begin{enumerate}
 \item
 $I_{k_0}\subseteq M$.
 \item $\varphi(I_{k_0})\subseteq I_{k_0}$ if and only if $p|F_{i,i_0}$
 for  $i\ges k_0$, and $p|F_{ij}$ for $i\ges k_0$ and $j<k_0$.
 \end{enumerate}
\end{enumerate}
\end{lemma}

\begin{proof}
We prove item (\ref{llabinv1}), leaving item (\ref{llabinv2}), which is similar,
to the reader. Since 
$$
z_j =
\left\{
\begin{array}{ll}
y_j  & \mbox{ if }\,\,j\les k_0\\
py_j+f_jy_{i_0}  & \mbox{ if }\,\,k_0\les j< i_0\\
y_j  & \mbox{ if }\,\,j= i_0\\
py_j & \mbox{ if }\,\,j> i_0,
\end{array}
\right.
$$
we have
$I_{k_0}\subseteq M$.
Observe that 
$\varphi(I_{k_0})\subseteq I_{k_0}$ if and only if 
$\varphi(z_j)\in I_{k_0}$ for all $j$. Since 
$$
\varphi(z_j) =
\left\{
\begin{array}{ll}
\sum_i F_{ij}x_i  & \mbox{ if }\,\,j< k_0\\[4pt]
\sum_i (pF_{ij}+f_j F_{i,i_0})x_i  & \mbox{ if }\,\,k_0\les j< i_0\\[4pt]
\sum_i F_{i,i_0}x_i & \mbox{ if }\,\,j= i_0\\[4pt]
\sum_i pF_{ij}x_i  & \mbox{ if }\,\,j> i_0,
\end{array}
\right.
$$
item (\ref{llabinv1}) follows.
\end{proof}

\vspace{5mm}
\noindent
\textbf{Proof of Proposition \ref{phdn}.}
Let $L=L(a,b)$, $M\subseteq L$ be a subalgebra of index $p$,
and $\varphi:M\rar L$ be a homomorphism of algebras.
We have to show that there exists a non-trivial $\varphi$-invariant ideal
$I$ of $L$. Observe that any $b_i$ is non-zero, and that $\mr{dim}[L,L]=d$. 
Moreover,
$v(a_i)\ges 1$ for all $i$, and any $I_\alpha$ 
is a non-trivial ideal of $L$ (see Remark \ref{rlabid} and Lemma \ref{llabid}).
Let $\bm{x}=(x_0,...,x_d)$ be the canonical basis of $L$,
and let $\bm{y}=(y_0,...,y_d)$ be a basis of $M$ in one of the forms given in
Remark \ref{rmbases}. The bases $\bm{x}$
and $\bm{y}$ are good bases for $L$ and $M$ respectively (cf. Lemma \ref{lsmetB}).
Let $A$ be the matrix of $L$ with respect to $\bm{x}$
(cf. Section \ref{smet}), and observe
that it is equal to the matrix $A$ of Lemma \ref{lpFij}.
Let $y_\beta=\sum_\alpha U_{\alpha\beta}x_\alpha$, and let
$\varphi(y_\beta) =\sum_\alpha F_{\alpha\beta}x_\alpha$.
By Lemma \ref{lmetmor}(\ref{lmetmor2}), $F_{0i}=0$ for all $i$.
The proof is completed below by considering each one of the three cases
of Remark \ref{rmbases}. For the last two cases, in order to apply
Lemma \ref{lpFij}, we have to make some observations.
In those cases, $U_{00}=0$ and the $d\times d$ matrix $U=(U_{ij})$
has the format of the one of Lemma \ref{lpFij}. The matrix
of $M$ with respect to $\bm{y}$ is $B=U^{-1}AU$; moreover,
$FB=F_{00}AF$ (Lemma \ref{lmetmor}(\ref{lmetmor3})).
An easy computation gives $F\widehat{U}^TA= F_{00}AF\widehat{U}^T$,
where $\widehat{U}$ is the cofactor matrix of $U$.
\begin{enumerate}
\item For case (\ref{ihdpa1}) of Remark \ref{rmbases}
we take $I=I_0$, which is invariant by Corollary \ref{cFeqs}(\ref{cFeqs2}).
\item For case (\ref{ihdpa2}) of Remark \ref{rmbases}
we take $I=I_{i_0}$, which is invariant by
Lemma \ref{lpFij}(\ref{lpFij2}) and Lemma \ref{llabinv}(\ref{llabinv2}).
\item For case (\ref{ihdpa3}) of Remark \ref{rmbases}
we take $I=I_{k_0}$, which is invariant by
Lemma \ref{lpFij}(\ref{lpFij1}) and Lemma \ref{llabinv}(\ref{llabinv1}).
\ep
\end{enumerate}


\subsection{Strongly hereditarily self-similar Lie lattices}\label{sechss}

\begin{definition}\label{dhss}
Let $L$ be a $\bb{Z}_p$-Lie lattice, and let $k\in\bb{N}$.
\begin{enumerate}
\item \label{dhss1}
$L$ is
\textbf{hereditarily self-similar of index $\bm{p}^k$}
if and only if any finite-index subalgebra of $L$
is self-similar of index $p^k$.
\item \label{dhss2}
$L$ is 
\textbf{strongly hereditarily self-similar of index $\bm{p}^k$}
if and only if $L$ is self-similar of index $p^k$ and 
any non-zero subalgebra of $L$
is self-similar of index $p^k$.
\end{enumerate}
\end{definition}

The main result of this section is as follows,
and the proof of the theorem will be given at the end of the section.

\begin{definition}
Let $d\ges 2$ be an integer, and let $a\in\bb{Z}_p$. 
We define an antisymmetric $d$-dimensional $\bb{Z}_p$-algebra $L^d(a)$
as follows. As a $\bb{Z}_p$-module, $L^d(a)=\bb{Z}_p^{d}$. Denoting by
$(x_0,...,x_{d-1})$ the canonical basis of $\bb{Z}_p^d$,  
the bracket of $L^d(a)$ is induced by
the commutation relations $[x_i,x_j]=0$ and $[x_0,x_i]=ax_i$,
where $i,j$ take values in $\{1,...,d-1\}$.
\end{definition}

\begin{theorem}\label{tshssp}
Assume $p\ges 3$.
Let $d\ges 2$ be an integer, and let $L$ be a solvable $\bb{Z}_p$-Lie lattice 
of dimension $d$ that is  strongly
hereditarily self-similar of index $p$.
Then $L\simeq L^d(p^s)$ for a unique $s\in\bb{N}\cup\{\infty\}$
(with $p^\infty := 0$).
\end{theorem}

\vspace{0mm}

Before proving the theorem we provide some examples and make some remarks
on hereditarily self-similar Lie lattices.

\begin{remark}\label{rhered}
Let $L$ be a $\bb{Z}_p$-Lie lattice. Clearly, if $L$ is strongly 
hereditarily self-similar of index $p^k$ then $L$ is
hereditarily self-similar of index $p^k$. From 
\cite[Remark 2.2]{NS2019} it follows that if $L$ has 
dimension 1 or 2 then $L$ 
is strongly hereditarily self-similar of index $p^k$
for all $k\ges 1$.
Consequently,
if $L$ has dimension 3 and
$L$ is hereditarily self-similar of index $p^k$ then
$L$ is strongly hereditarily self-similar of index $p^k$.
Proposition \ref{hereddim3v2} below classifies, for $p\ges 3$, 
the 3-dimensional Lie lattices that are hereditarily self-similar of index $p$.
\end{remark}

\begin{proposition}\label{heranyindex}
Let $m\ges 1$ be an integer, and let $L$ be a $3$-dimensional
solvable $\bb{Z}_p$-Lie lattice. Then $L$ is strongly hereditarily self-similar of index
$p^{2m}$.
\end{proposition}
\begin{proof}
The proposition follows from 
Corollary \ref{csigmap2} and Remark \ref{rhered}.
\end{proof}

\vspace{3mm}
 
Proposition \ref{heranyindex} and 
\cite[Proposition 3.1]{SnopceJGT2016} 
have a consequence that
we find worth to state explicitly. We recall that, by definition, 
two Lie lattices $L_1$ and $L_2$
are incommensurable if there are no finite-index subalgebras $M_1\subseteq L_1$ 
and $M_2\subseteq L_2$ such that $M_1\simeq M_2$.

\begin{corollary}
There exists a set $\cl{H}$ of the cardinality of the continuum such that
any element of $\cl{H}$ is a $\bb{Z}_p$-Lie lattice that is 
strongly hereditarily self-similar of index
$p^{2m}$ for each  $m\ges 1$, and such that any two distinct elements of $\cl{H}$
are incommensurable.
\end{corollary}

The next results are interesting on their own and they are a preparation
for the proof of Theorem \ref{tshssp}.

\begin{remark}\label{rlda2}
We list some properties of $L= L^d(a)$ that the reader may easily prove.
The Lie lattice $L$  belongs to the class discussed
in Section \ref{smet}; in particular, 
$L$ is a Lie lattice and $\delta_2(L)=\{0\}$.
We have $L^d(a)\simeq L^e(b)$ if and only if $d=e$ and $v_p(a)=v_p(b)$;
moreover,
$L$ is abelian if and only if $a=0$. If $a\neq 0$ then 
$\mr{iso}_L [L,L]=\gen{x_1,...,x_{d-1}}$.
Any submodule of $L$ is a subalgebra,
and any 2-generated subalgebra of $L$ has dimension at most 2.
Finally, note that $L_0 = L^3(0)$,
$L_1(s)=L^3(p^s)$, and $L_{\ref*{C6}}(a) = L^3(a)$ (see Definition \ref{dlielat}).
\end{remark}

\begin{lemma}\label{llda}
Let $d\ges 2$ be an integer, let $a\in\bb{Z}_p$,
and let $M$ be a subalgebra of $L^d(a)$ of dimension $e\ges 2$.
Then $M\simeq L^e(p^sa)$ for some $s\in\bb{N}\cup\{\infty\}$.
\end{lemma}

\begin{proof}
Denote $L=L^d(a)$ and recall that $L$ is endowed with the
basis $(x_0,...,x_{d-1})$.
Let $J_L = \mr{iso}_L[L,L]$. If $M\subseteq J_L$ then one
takes $s=\infty$. 
Assume $M\not\subseteq J_L$. Hence, $a\neq 0$ and
$L/J_L\simeq\bb{Z}_p$,
generated by the class of $x_0$. Let $\varphi:M\rar L/J_L$
be the canonical map. Then $\varphi(M)$ is non-zero,
hence, there exists $x\in M$ such that the class of $x$ in $L/J_L$
is a basis of $\varphi(M)$ over $\bb{Z}_p$.
Also, $[L/J_L:\varphi(M)]=p^s$ for some $s\in\bb{N}$.
Let $x=cx_0+\sum_j c_jx_j$, where the index takes values in $\{1,...,d-1\}$.
Observe that $v_p(c) = s$.
One proves that $M=\gen{x}\oplus (M\cap J_L)$,
from which the conclusion $M\simeq L_e(ca)\simeq L_e(p^sa)$ follows.
\end{proof}

\begin{proposition}\label{pl1dssp}
Let $d\ges 2$ and $k\ges 1$ be integers, and let $a\in\bb{Z}_p$.
Then: 
\begin{enumerate}
\item \label{pl1dssp1}
$L^d(a)$ is self-similar of index $p^k$.
\item \label{pl1dssp2}
$L^d(a)$ is strongly hereditarily self-similar of index $p^k$.
\end{enumerate}
\end{proposition}

\begin{proof}
Item (\ref{pl1dssp2}) is a consequence of item (\ref{pl1dssp1})
and Lemma \ref{llda}. We prove (\ref{pl1dssp1}).
For $d=2$ see Remark \ref{rhered}. Assume $d\ges 3$.
Let $L = L^d(a)$, and let
$M = \langle x_0, p^kx_1, x_2, ..., x_{d-1} \rangle$. 
Then $M$ is a subalgebra of $L$ of index $p^k$. 
The module homomorphism $\varphi: M \to L$ determined by
$\varphi(x_0)=x_0$, $\varphi(p^k x_1)=x_2$, $\varphi(x_i)=x_{i+1}$ for
$2\les i < d-1$, and $\varphi(x_{d-1})=x_1$
is a homomorphism of algebras.
We prove that $\varphi$ is a simple.
Indeed, the 
intersection of the domains of the powers of $\varphi$
is $D_\infty =\langle x_0 \rangle$. Let $I$ be a non-trivial ideal of $L$.
Similarly to what has been done in the proof of Proposition \ref{pabelianid},
one shows that $L$ is not $\varphi$-invariant by proving the existence of 
$w\in I$ such that $w\not\in D_\infty$. 
\end{proof}

\begin{proposition}\label{hereddim3v2}
Assume $p\ges 3$, and let
$L$ be a 3-dimensional $\bb{Z}_p$-Lie lattice.
The following are equivalent.
\begin{enumerate}
\item \label{h1}
$L$ is hereditarily self-similar of index $p$.
\item \label{h2}
$L$ is isomorphic either to $L_{\ref*{C0}}$ or to $L_{\ref*{C1}}(s)$
for some $s\in\bb{N}$.
\end{enumerate}
\end{proposition}

\begin{proof}
Remark \ref{rlda2} and Proposition \ref{pl1dssp}(\ref{pl1dssp2})
show that the implication `(\ref{h2})$\Rightarrow$(\ref{h1})' holds
even in greater generality than stated here.
For the other implication, we assume that (\ref{h2}) does not hold,
and 
we show that there exists a finite-index 
subalgebra $M$ of $L$ that is not self-similar of index $p$.
We divide the proof into two parts according to whether $L$ 
is solvable or unsolvable.

Assume that $L$ is solvable. The following observations are enough 
to cover all the cases (cf. Remark \ref{rclassif}).
If $\mr{dim}[L,L]=1$ then $L$ itself is not self-similar of index $p$
(Proposition \ref{pll1nssp}). If $r\ges 1$ then 
$M = \gen{x_0,px_1,x_2}$ is a subalgebra of 
$L_2(s,r,c)$ and $M\simeq L_{\ref*{C2}}(s,r-1,p^2c)$.
If $r\ges 2$ then $M$ is not self-similar of index $p$ by 
Proposition \ref{pl2nssp}. If $r=1$
one shows that $L_{\ref*{C2}}(s,0,p^2c)\simeq 
L_{\ref*{C7}}(s,1,4^{-1}(p^2c-1))$,
so that $M$ is not self-similar of index $p$ by item (\ref{pl75}) of 
Proposition \ref{pl7nssp}.
Now, let $L=L_{\ref*{C7}}(s,a,c)$ with $c\neq 0$. Observe that $pL$ 
is subalgebra of $L$ and that $pL\simeq L_{\ref*{C7}}(s+1,a,c)$; hence, we can assume 
that $s\ges 1$. Then $M = \gen{x_0,px_1,x_2}$ is a subalgebra of $L$
and $M\simeq L_{\ref*{C7}}(s-1,pa,p^2c)$,
so that $M$ is not self-similar of index $p$ by
item (\ref{pl71}) of Proposition \ref{pl7nssp}.

Now, assume that $L$ is unsolvable. There exists a basis $(x_0,x_1,x_2)$
of $L$ such that $[x_i,x_{i+i}]=a_{i+2}x_{i+2}$, where the index $i$
is interpreted in $\bb{Z}/3\bb{Z}$, and the 
$a_i$'s are non-zero
$p$-adic integers with $v_p(a_0)\les v_p(a_1)\les v_p(a_2)$;
see \cite[Proposition 2.7]{NS2019}.
It is not difficult to see that one can choose $k_0,k_1,k_2\in\bb{N}$
such that, defining $y_i=p^{k_i}x_i$, one has $[y_i,y_{i+i}]=b_{i+2}y_{i+2}$,
where 
the $b_i$'s are non-zero
$p$-adic integers and $v_p(b_0)< v_p(b_1)<v_p(b_2)$.
Hence,
$M = \gen{y_0,y_1,y_2}$ is a subalgebra of $L$ that is not self-similar
of index $p$ by \cite[Theorem 2.32]{NS2019}. 
\end{proof}


\vspace{5mm}

\noindent
\textbf{Proof of Theorem \ref{tshssp}.}
Uniqueness of $s$ is easy to prove (cf. Remark \ref{rlda2}). 
The proof of existence is by induction on $d$.
For $d=2$ the theorem is easily proven, while for $d=3$ it follows
from Proposition \ref{hereddim3v2} and Remark \ref{rlda2}. 
For the induction step,
let $d\ges 4$ and assume that the theorem holds with $d'$ in place of $d$,
where $d'<d$. 
Let $\mathcal{L} = L \otimes_{\bb{Z}_p} \mathbb{Q}_p$. 
Since $\mathcal{L}$ is a solvable Lie algebra over a field of characteristic 0, 
Lie's theorem implies that the $\mathbb{Q}_p$-Lie algebra 
$[\mathcal{L}, \mathcal{L}]$ is nilpotent. 
Hence, the  $\bb{Z}_p$-Lie lattice $[L, L]$ is nilpotent as well. 

We prove that $[L,L]$ is abelian.
Denote temporarily $M=[L,L]$, and 
assume by contradiction that $M$ is not abelian.
Hence, $M$ is a non-abelian
nilpotent Lie lattice. Let $c$ be the nilpotency class of $M$;
then $c\ges 2$. We claim that there exists $x,y\in M$ such that
$[x,y]\neq 0$ and $[x,y]\in Z(M)$ (the center of $M$).
Indeed, $\{0\} \neq \gamma_{c-1}(M)\subseteq Z(M)$.
Hence, there exist $x\in M$ and $y\in\gamma_{c-2}(M)$
such that $[x,y]\neq 0$.
Since $[x,y]\in\gamma_{c-1}(M)$, it follows that $[x,y]\in Z(M)$,
and the claim is proven.
Let $N$ be the subalgebra generated by $x$ and $y$. 
Then $N$ is a nilpotent non-abelian subalgebra
of $L$ with $\mr{dim}[N,N]=1$. The dimension of $N$ is either 2 or 3.
Since no non-abelian Lie lattice of dimension 2 is nilpotent, we have
$\mr{dim}\, N = 3$. Hence, $N$ is not self-similar of index $p$
by Proposition \ref{pll1nssp}, a contradiction.

Let $m=\mr{dim}[L,L]$. Note that $m<d$, since otherwise $L$ would not be solvable.
If $m=0$ ($L$ abelian) then one takes $s=\infty$. Assume $m> 0$ ($L$ not abelian).
Let $J=\mr{iso}_L[L,L]$, which is an isolated abelian ideal of $L$.
Hence, $\mr{dim}\,J= m$,
and there exists a basis $(x_1,...,x_{d-m},y_1,...,y_m)$ of
$L$ such that $(y_1,...,y_m)$ is a basis of $J$.
Let greek indices take values in $\{1,...,d-m\}$,
and latin indices take values in $\{1,...,m\}$.
We have $[y_i,y_j]=0$, and any commutator in $L$
is a linear combination of the $y_i$'s.
Let $M_\alpha = \gen{x_\alpha,y_1,...,y_m}$.
Then $M_\alpha$ is a subalgebra of $L$ of dimension $m+1\ges 2$.
For $z\in J$, $z\neq 0$, define $M'_z$ to be the subalgebra of $L$
generated by $x_1$ and $z$. Observe that $M'_z$ has dimension $n_z\ges 2$.
Moreover, observe that all $M_\alpha$'s and $M'_z$'s are solvable and 
strongly hereditarily self-similar of index $p$. 
We divide the proof in two cases.
\begin{enumerate}
\item Case $m<d-1$. Then $m+1<d$ and $M_\alpha\simeq L^{m+1}(p^{s_\alpha})$ 
for some $s_\alpha\in\bb{N}\cup\{\infty\}$.
Since $\gen{x_\alpha,y_i}$ is a subalgebra of $M_\alpha$,
$[x_\alpha,y_i]=c_{\alpha i}y_i$ for some $c_{\alpha i }\in\bb{Z}_p$. 
By contradiction,
assume $c_{\alpha i}\neq c_{\alpha j}$ for some $i,j$. 
Since $\gen{x_\alpha, y_i+y_j}$ is a subalgebra of $M_\alpha$,
$[x_\alpha,y_i+y_j]=c_{\alpha i}(y_i+y_j)+(c_{\alpha j}-c_{\alpha i})y_j
\in\gen{x_\alpha, y_i+y_j}$, which is a contradiction.
It follows that $[x_\alpha,y_i]=c_\alpha y_i$ for all indices $i$
and some $c_\alpha\in\bb{Z}_p$ with $v_p(c_\alpha)=s_\alpha$.
Observe that $d-m\ges 2$, and that $c_{\alpha_0}\neq 0$ for
some $\alpha_0$.
 \begin{enumerate}
 \item 
 Case $[x_\alpha,x_\beta]=0$ for all $\alpha,\beta$.
 Let $N=\gen{x_{\alpha_0},x_{\alpha_1},y_1}$ with $\alpha_1\neq \alpha_0$.
 Then $N$ is a subalgebra of $L$ of dimension $3$, and $\mr{dim}[N,N]=1$.
 Hence, $N$ is not self-similar of index $p$, a contradiction.
 \item 
 Case $[x_{\beta_0},x_{\beta_1}]\neq 0$ for some $\beta_0,\beta_1$.
 Let $z = [x_{\beta_0},x_{\beta_1}]$, and let $N=\gen{x_{\beta_0},x_{\beta_1},z}$.
 Then $N$ is a subalgebra of $L$ of dimension $3$, and $\mr{dim}[N,N]=1$
 (observe that $[x_{\beta_{j}},z] = c_{\beta_j}z$).
 Hence, $N$ is not self-similar of index $p$, a contradiction.
 \end{enumerate}
 
\item Case $m=d-1$. Recall the notation $n_z = \mr{dim}\,M'_z$.
 \begin{enumerate}
 \item 
 Case $n_z=d$ for some $z$. Let $M=M'_z$ and $J_M=\mr{iso}_M[M,M]$.
 Observe that $\mr{dim}\,J_M = d-1$. Define by recursion $z_1 = z$
 and $z_{i+1}=[x,z_i]$ for $i\ges 1$. One can show that
 $J_M = \gen{z_i:i\ges 1}$. We claim that $\{z_1,...,z_{d-1}\}$
 is a basis of $J_M$. Indeed, denoting by $\bar{w}$ the residue 
 of $w\in J_M$ in $J_M/pJ_M$, we show that 
 $\{\bar{z}_1,...,\bar{z}_{d-1}\}$ is linearly independent over 
 $\bb{F}_p=\bb{Z}_p/p\bb{Z}_p$. If it was not independent,
 some $\bar{z}_{j_0}$ would be a linear combination of 
 $\bar{z}_1,...,\bar{z}_{j_0-1}$, and one could prove
 (from the recursive definition of the $z_i$'s)
 that any $\bar{z}_i$, $i\ges j_0$, 
 would be such a linear combination, so that the dimension
 of $J_M/pJ_M$ over $\bb{F}_p$ would be less than $d-1$, a contradiction.
 The claim that $\{z_1,...,z_{d-1}\}$ is a basis of $J_M$ over $\bb{Z}_p$ follows,
 and from it we get a basis $\{x,z_1,...,z_{d-1}\}$ of $M$
 where $[z_i,z_j]=0$, $[x,z_i] = z_{i+1}$ for $i<d-1$,
 and $[x,z_{d-1}]=\sum_{j=1}^{d-1}a_jz_j$ for some $a_j\in\bb{Z}_p$.
 We claim that $a_1\neq 0$.
 By contradiction, assume $a_1=0$.
 Then $N:=\gen{x,z_2,...,z_{d-1}}$ is a subalgebra of $L$
 of dimension $d-1\ges 3$.
 Moreover, $N$ is solvable 
 and strongly hereditarily self-similar of index $p$.
 Thus, there exists $s\in\bb{N}\cup\{\infty\}$ 
 such that $N\simeq L^{d-1}(p^s)$. Then $\gen{x,z_2}$
 is a subalgebra of $N$, a contradiction (since $d\ges 4$).
 Hence, $a_1\neq 0$. By Corollary \ref{cnsspsub},  
 there exists a non-zero subalgebra of $M$ that is not self-similar of index $p$,
 which gives a contradiction.
 \item 
 Case $n_{z}<d$ for all $z$. Then $M'_{y_i}\simeq L^{n_{y_i}}(p^{s_i})$ for some 
 $s_i\in\bb{N}\cup\{\infty\}$ (for all $i$). Hence, $\gen{x_1,y_i}$
 is a subalgebra of $M'_{y_i}$, and  so $[x_1,y_i]=b_iy_i$
 for some $b_i\in\bb{Z}_p$. Assume by contradiction
 that $b_{j_0}\neq b_{j_1}$ for some $j_0,j_1$. Let $z_0 = y_{j_0}+y_{j_1}$.
 Then $M'_{z_0} \simeq L^{n_{z_0}}(p^t)$. 
 On the other hand, $[x_1,z_0] = b_{j_0}z_0 + (b_{j_1}-b_{j_0})y_{j_1}$ 
 yields that the 2-generated algebra $L^{n_{z_0}}(p^t)$ 
 has dimension greater than 2, 
 which is a contradiction.
 Thus, $[x,y_i]=by_i$ for all indices $i$ and some $b\in\bb{Z}_p$ with
 $b\neq 0$. Hence, $L\simeq L^d(b)\simeq L^d(p^s)$, where $s=v_p(b)$. \ep
 \end{enumerate}
\end{enumerate}

\section{Results on groups}\label{sresgp}

In this section we prove the main theorems of the paper, stated in the introduction. 
Essentially, the proofs follow from the results on Lie lattices of Section 
\ref{seclie}, and from the following slightly generalized version
of \cite[Proposition A]{NS2019}.

\begin{proposition}  \label{Prop-NS2}
Let $G$ be a torsion-free $p$-adic analytic pro-$p$ group.
Assume that any closed subgroup of $G$ is saturable,
and that any 2-generated closed subgroup of $G$ has dimension at most $p$.
Let $L_G$ be the $\bb{Z}_p$-Lie lattice
associated with $G$, and assume that
any 2-generated subalgebra of $L_G$ has dimension
at most $p$. Then, for all $k\in\bb{N}$, the following holds.
\begin{enumerate}
\item $G$ is a self-similar group of index $p^k$
if and only if $L_G$ is a self-similar Lie lattice of index $p^k$.
\item $G$ 
is hereditarily self-similar of index $p^k$
(respectively, strongly hereditarily self-similar of index $p^k$)
if and only if $L_G$ is 
hereditarily self-similar of index $p^k$
(respectively, strongly hereditarily self-similar of index $p^k$).
\end{enumerate}
\end{proposition}

\begin{proof}
The proposition follows from 
Lazard's correspondence \cite{Laz65},
Theorem E of \cite{GSKpsdimJGT},
the argument proving $[G:D]=[L_G:L_D]$ in the proof of \cite[Theorem 3.1]{NS2019},
and \cite[Proposition 1.3]{NS2019}.
\end{proof}

\begin{remark}\label{rgsk2}
Let $G$ be a torsion-free $p$-adic analytic pro-$p$ group.
If $G$ is saturable and $\mr{dim}(G)\les p$ then
the hypotheses of Proposition \ref{Prop-NS2} are satisfied;
if $\mr{dim}(G)<p$ then the same conclusion holds without assuming 
a priori that $G$ is saturable \cite[Theorem A]{GSKpsdimJGT}. 
We will also use the fact that if $G$ is saturable and $L_G$
is the associated $\bb{Z}_p$-Lie lattice then $G$ is solvable
if and only if $L_G$ is solvable \cite[Theorem B]{GSpsat}.
\end{remark}

\begin{remark}\label{rghered}
This remark is the analogue of Remark \ref{rhered}
in the context of groups.
Let $G$ be a finitely generated pro-$p$ group.
For $k\in\bb{N}$, if $G$ is strongly 
hereditarily self-similar of index $p^k$ then $G$ is
hereditarily self-similar of index $p^k$.
Assume, moreover, that $G$ is torsion-free and $p$-adic analytic.
From \cite[Proposition 1.5]{NS2019} it follows that if $\mr{dim}(G)=1,2$
then $G$ is strongly hereditarily self-similar of index $p^k$
for all $k\ges 1$.
Consequently, if $G$ has dimension 3 and
$G$ is hereditarily self-similar of index $p^k$ then
$G$ is strongly hereditarily self-similar of index $p^k$.
\end{remark}

\begin{proposition}\label{pgstrong2m}
Let $m\ges 1$, and let $G$
be a 3-dimensional solvable torsion-free $p$-adic analytic pro-$p$ group.
Assume either `$\,p\ges 5$' or `$\,p=3$ and $G$ is saturable'.
Then $G$ is strongly hereditarily self-similar of index $p^{2m}$.
\end{proposition}

\begin{proof}
The proposition follows from Propositions \ref{Prop-NS2} and \ref{heranyindex}.
\end{proof}

\vspace{3mm}

\noindent
\textbf{Proof of Theorem \ref{tmaingroups}.} 
Let $L$ be the $\bb{Z}_p$-Lie lattice associated with $G$.
Then $L$ is a residually nilpotent 3-dimensional solvable Lie lattice
\cite[Theorem B]{GSKpsdimJGT}.
From Corollary \ref{csigmap2}, $L$ is self-similar of index $p^2$.
Hence, by Proposition \ref{Prop-NS2}, $G$ is self-similar of index $p^2$.
The statement on self-similarity of index $p$ follows from
Proposition \ref{Prop-NS2}, Theorem \ref{talgpodd} and Remark \ref{rclassif}.
\ep

\vspace{5mm}

\noindent
\textbf{Proof of Theorem \ref{tmainD}.}
The theorem follows from Remark \ref{rghered}, 
Proposition \ref{Prop-NS2}, and Proposition \ref{hereddim3v2}.
\ep

\begin{remark}\label{rsatp3}
A similar result to Theorem \ref{tmaingroups} holds for $p=3$.
Let $G$ be a 3-dimensional solvable \textit{saturable} $3$-adic analytic pro-$3$ group.
Then $G$ is self-similar of index $9$. Let $L$ be the $\bb{Z}_3$-Lie lattice
associated with $G$.
Then $G$ is self similar of index $3$ if and only if $L$ is isomorphic to a Lie lattice
appearing in the list of Theorem \ref{talgpodd}.
\end{remark}

\begin{remark}  \label{rhere}
Let $G$ be one of the groups in the 
list below, where $d$ is an integer. Observe that
this list extends the one appearing in the statement
of Theorem \ref{tmainresult} (here there 
is no assumption $p>d$).
  \begin{enumerate}
  \item For $d\ges 1$, the abelian pro-$p$ group $\mathbb{Z}_p^d$;
  \item 
  For $d\ges 2$, the metabelian pro-$p$ group
  $G^d(s) = \bb{Z}_p\ltimes \bb{Z}_p^{d-1}$, 
  where
  the canonical generator of 
  $\bb{Z}_p$ acts on $\bb{Z}_p^{d-1}$
  by  multiplication by the scalar $1+p^s$ for some integer $s$ such that 
  $s\ges 1$ if $p\ges 3$, and $s\ges 2$ if $p=2$.
  \end{enumerate}
Then $G$ is a uniformly powerful $p$-adic analytic pro-$p$
group of dimension $d$. 
Let $L_G$ be the $\bb{Z}_p$-Lie lattice associated with $G$.
Observe that if $G$ is abelian then $L_G\simeq L^d(0)$,
while if $G=G^d(s)$ then $L_G\simeq L^d(p^s)$.
One can show that any subgroup of $G$
generated by  two elements is powerful.
It follows that any closed subgroup of $G$
is uniformly powerful, hence, saturable.
Clearly, any 2-generated closed subgroup 
of $G$ has dimension at most 2.
\end{remark}

\begin{proposition}\label{here2}
Let $k\ges 1$ be an integer, and let $G$ be a group isomorphic to 
one of the groups in the list of Remark \ref{rhere}. Then
$G$ is strongly hereditarily self-similar of index $p^k$.
\end{proposition}

\begin{proof}
If $d:=\mr{dim}(G)=1$ then
$G\simeq \bb{Z}_p$ and the result is clear.
Assume $d\ges 2$. 
The result follows from 
Remark \ref{rhere}, Remark \ref{rlda2}, 
Proposition \ref{Prop-NS2} and Proposition \ref{pl1dssp}. 
\end{proof}

\vspace{5mm}

Under the assumption that $p>\mr{dim}(G)$ we can prove the converse of 
Proposition \ref{here2}, which is the main result of the paper.

\vspace{3mm}

\noindent
\textbf{Proof of Theorem \ref{tmainresult}.} 
The `if' part follows from Proposition \ref{here2}.
For the `only if' part, if $d=1$ then $G\simeq \bb{Z}_p$.
Assume $d\ges 2$. Observe that in this case $p\ges 3$.
By Remark \ref{rgsk2} we can apply Proposition \ref{Prop-NS2}.
Let $L_G$ be the $\bb{Z}_p$-Lie lattice associated with $G$, which is residually nilpotent.
From Theorem \ref{tshssp}, $L_G\simeq L^d(p^s)$ for some $s\in\bb{N}\cup\{\infty\}$,
while from residual nilpotency we deduce that $s\ges 1$.
Now, the theorem follows from Remark \ref{rhere}.
\ep

\vspace{5mm}

Assume that $p$ is odd, and let $K$ be a field
that contains a primitive $p$-th root of unity
(necessarily, $K$ has characteristic different from $p$).
In \cite{Ware92}, Roger Ware proved
that if $G_K(p)$ is finitely generated and it  does not contain a non-abelian free pro-$p$ subgroup, then  $G_K(p)$  is either a free abelian pro-$p$ group of finite rank, or it is isomorphic to $G^d(s)$ for some  
integers $d \ges 2$ and $s \ges 1$.  
In particular, the same conclusion holds if $G_K(p)$ is solvable or $p$-adic analytic.
Indeed, Ware proved this result under the additional assumption that $K$ contains a primitive $p^2$-th root of unity, and conjectured that the result should be true without this assumption. The conjecture was proved by Quadrelli in \cite{Quad14}. 
As a direct consequence of Proposition \ref{here2} and the result of Ware, we have the following.

\begin{proposition}  \label{maximal}
Assume $p\ges 3$, and let $K$ be a field that 
contains a primitive $p$-th root of unity.   Suppose that $G_K(p)$ is a non-trivial finitely generated pro-$p$ group that does not contain a non-abelian free pro-$p$ subgroup.
Then $G_K(p)$ is strongly hereditarily self-similar of index $p$.
\end{proposition} 

Conversely, assuming $p$-odd, it is shown in \cite{Ware92} that
any group in the list of Remark \ref{rhere}
is isomorphic to $G_K(p)$ for some field $K$
that contains a primitive $p$-th root of unity. We recall the construction
of $K$ for the non-abelian groups $G^d(s)$, in which case
$d\ges 2$ and $s \ges 1$. 
Let $r$ be a prime with 
$r \equiv_p 1$, 
and let $F = \mathbb{F}_r(\omega_s)$, where $ \mathbb{F}_r$ 
is a finite field with $r$ elements and $\omega_s$ is a primitive $p^s$-th root of unity. Then one may take $K = F((x_1))  \cdots  ((x_{d-1}))$, the field of iterated formal Laurent series. 

\vspace{5mm}

\noindent
\textbf{Proof of Theorem \ref{tmainB}.} 
For $p>2$ the result follows from Theorem \ref{tmainresult} and the above discussion. When $p=2$, we observe that $G_{\mathbb{F}_q}(2) \simeq \bb{Z}_2$ for any finite field $\mathbb{F}_q$ with $q$ elements; this  follows from the well-known fact that the absolute Galois group of  $\mathbb{F}_q$ is isomorphic to 
$ \widehat{\bb{Z}} = \prod_{r}{\bb{Z}_r}$, where the product ranges over
all primes $r$.

\vspace{5mm}

As mentioned in the Introduction, 
during the last decade the groups listed in Theorem \ref{tmainresult}
have been object of study. We recall the related results,
and we complement them with the results of this paper.
A pro-$p$ group $G$ is said to have a constant generating number on open subgroups if $d(H) = d(G)$ for all open subgroups $H$ of $G$,
where $d(G)$ is the minimum number of elements of
a topological generating set for $G$. 
Pro-$p$ groups with constant generating number on open subgroups were classified by Klopsch and Snopche in \cite{KSjalg11}.  On the other hand, a Bloch-Kato pro-$p$ group is a pro-$p$ group $G$ with the property that the $\mathbb{F}_p$-cohomology ring of every closed subgroup of $G$ is quadratic. 
In \cite{Quad14}, Quadrelli described explicitly all finitely generated Block-Kato pro-$p$ groups that do not contain a free non-abelian pro-$p$ group. Finally,  a pro-$p$ group $G$ is said to be hereditarily uniform if every open subgroup of $G$ is uniform. Hereditarily uniform pro-$p$ groups were classified by Klopsch and Snopche in \cite{KSquart14}.  The results of Klopsch, Snopche and Quadrelli (\cite[Corollary 2.4]{KSjalg11}, \cite[Corollary 1.13]{KSquart14} and \cite[Theorem B]{Quad14}) together with Theorem \ref{tmainB} yield the following.

\begin{theorem}\label{thBext}
Let $G$ be a non-trivial solvable torsion-free $p$-adic analytic pro-$p$ group, and suppose that $p > \textrm{dim}(G)$. Then the following are equivalent. 
  \begin{enumerate}
  \item $G$ is strongly hereditarily self-similar of index $p$.
  \item $G$ is isomorphic to the maximal pro-$p$ Galois group of some field that contains a primitive $p$-th root of unity.
   \item  $G$ has constant generating number on open subgroups.
    \item $G$ is a Bloch-Kato pro-$p$ group.
     \item  $G$ is a hereditarily uniform pro-$p$ group.
    \end{enumerate} 
\end{theorem}


\section{Open problems}\label{sopen}

This paper deals with as-yet-unexplored directions about self-similar groups, 
so there are many interesting open problems that one may consider. The following two questions are quite natural.

\vspace{5mm}

\noindent \textbf{Question 1.} Classify the strongly hereditarily self-similar pro-$p$ groups of index $p$.

\vspace{5mm}

\noindent \textbf{Question 2.} Classify the  hereditarily self-similar pro-$p$ groups of index $p$.

\vspace{5mm}

All the examples of strongly hereditarily self-similar pro-$p$ groups of index $p$ that we know are $p$-adic analytic.

\vspace{5mm}

\noindent \textbf{Question 3.} Is there a finitely generated strongly hereditarily self-similar pro-$p$ group of index $p$ which is not $p$-adic analytic?

\vspace{5mm}

Let $K$ be a $p$-adic number field, that is, 
a finite extension of $\mathbb{Q}_p.$ It is well known (see \cite[Theorem 7.5.11]{NSW2008}) that if $K$ does not contain a primitive $p$-th root of unity then $G_K(p)$ is a free pro-$p$ group of finite rank. On the other hand, if $K$  contains a primitive $p$-th root of unity then $G_K(p)$ is a Demushkin group, that is, a 
Poincar\'e duality pro-$p$ group of dimension 2. 
Pro-$p$ completions of surface groups are also Demushkin groups. Inspired from Proposition \ref{maximal}, we raise the following questions.

\vspace{5mm}

\noindent \textbf{Question 4.} Does a free pro-$p$ group of finite rank admit a faithful self-similar action on a $p$-ary tree?

\vspace{5mm}

\noindent

\noindent \textbf{Question 5.} Does a Demushkin pro-$p$ group admit a faithful self-similar action on a $p$-ary tree?

\vspace{5mm}

\noindent

Note that an affirmative answer to Question 4 would imply that a free pro-$p$ group of finite rank is strongly hereditarily self-similar of index $p$. On the other hand, since  every open subgroup of a Demushkin group is also Demushkin,  an affirmative answer to Question 5 would imply that Demushkin groups are hereditarily self-similar of index $p$.  Moreover, since every infinite index subgroup of a Demushkin group is free pro-$p$, an affirmative answer to  both questions would imply that Demushkin groups are strongly hereditarily self-similar of index $p$. Note that if $G$ is a Demushkin group with $d(G) = 2$, then it is a torsion-free $p$-adic analytic pro-$p$ group of dimension 2, and therefore it is strongly hereditarily self-similar of index $p$. Thus Question 5 is open only for Demushkin groups $G$ with $d(G) > 2$.

\bibliographystyle{amsalpha}


\begin{footnotesize}

\providecommand{\bysame}{\leavevmode\hbox to3em{\hrulefill}\thinspace}
\providecommand{\MR}{\relax\ifhmode\unskip\space\fi MR }
\providecommand{\MRhref}[2]{%
  \href{http://www.ams.org/mathscinet-getitem?mr=#1}{#2}
}
\providecommand{\href}[2]{#2}

\end{footnotesize}

\end{document}